\documentclass[11pt]{article}
\usepackage[english]{babel}
\usepackage{amsfonts,amsmath,amssymb}
\usepackage{graphicx}
\usepackage{pst-all}
\usepackage{pstcol}
\usepackage{graphics}
\usepackage{fullpage}
\usepackage{enumerate}

\newtheorem{defi}{{\bf Definition}}[section]
\newtheorem{thr}[defi]{{\bf Theorem}}
\newtheorem{pro}[defi]{{\bf Proposition}}
\newtheorem{coro}[defi]{{\bf Corollary}}
\newtheorem{lem}[defi]{{\bf Lemma}}
\newcommand{\proof}{\noindent \textbf{Proof:~~}}
\newcommand{\Q}{\mathbb{Q}} % Conjuntos Racionales
\newcommand{\N}{\mathbb{N}} % Conjuntos Naturales
\newcommand{\base}{\mathcal{B}}   % B cursiva
\newcommand{\seq}{\N^{<\omega}}
\newcommand{\fin}{\hfill$\blacksquare$} %--------  Fin de una prueba
\newcommand{\up}{\upharpoonright} %--------------  simbolo de restriccion
\newcommand{\su}{\subseteq} %--------------------  Subconjunto
\def\cantor{2^{\N}}
\def\FIN{\sf FIN}
\def\fsigdel{F_{\sigma\delta}}

\title{Fr\'{e}chet Borel ideals with Borel orthogonal}
\author{Francisco Guevara\thanks{Partially supported by the CDCHTA grant \# SE-C-03-12-05 from Universidad de Los Andes, M\'erida, Venezuela.}  $\,$  and Carlos Uzc\'{a}tegui
}
\date{\today}

\begin{document}

\maketitle

\begin{abstract}
We study Borel ideals $I$ on $\N$ with the Fr\'{e}chet property such that
its orthogonal $I^\perp$ is also Borel (where $A\in I^\perp$ iff
$A\cap B$ is finite for all $B\in I$ and $I$ is Fr\'{e}chet if
$I=I^{\perp\perp}$). Let $\base$ be the smallest collection of
ideals on $\N$ containing the ideal of finite sets and closed
under countable direct sums and orthogonal. All ideals in $\base$
are Fr\'{e}chet, Borel and have Borel orthogonal. We show that
$\base$ has exactly $\aleph_1$ non isomorphic members. The family
$\base$ can be characterized as the collection of all Borel ideals
which are isomorphic to an ideal of the form $I_{wf}\!\!\up A$,
where $I_{wf}$ is the ideal on $\N^{<\omega}$ generated by the
wellfounded trees. Also, we show that $A\su \Q$ is scattered iff
$WO(\Q)\!\up A$ is isomorphic to an ideal in $\base$, where
$WO(\Q)$ is the ideal of well founded subsets of $\Q$. We use the
ideals in $\base$  to construct  $\aleph_1$ pairwise non homeomorphic
countable sequential spaces whose topology is analytic.
\end{abstract}


Subjclass[2000] {Primary and Secondary: 03E15, 03E05}.

\section{Introduction}\label{section1}
Given a collection $\cal A$ of subsets of $\N$, the orthogonal of
$\cal A$ is the following family of sets
\[
{\cal A}^{\perp}=\{B\su \N : (\forall A\in {\cal A})(A\cap B
\textrm{ is finite}) \}.
\]
In this paper we study some structural properties of  pairs
$(I,I^\perp)$ (also called, a gap) where $I$ is an ideal of subsets
of $\N$, that is to say, $I$ is a non empty collection of subsets of
$\N$ closed under finite unions and taking subsets of its elements.
The motivation of our work comes from some  results about definable
pairs of orthogonal families. The word definable refers to the descriptive
complexity of the family as a subset of the Cantor cube $2^\N$ (by the
usual identification of a subset of $\N$ with its characteristic function).
It is in that sense that we talk about Borel, analytic or co-analytic ideals. For
instance, if $I$ is Borel (or analytic), then $I^\perp$ is (at
most) co-analytic.

There has been growing interest in the study of pairs $({\cal
A},{\cal B})$ of orthogonal families (i.e, ${\cal A}\su {\cal
B}^\perp$) because of its natural connection with gaps in the
quotient algebra ${\cal P}(\N)/\mbox{Fin}$ and its applications to
problems in analysis and topology (see
\cite{AvilesStevo2011,AvilesStevo2012,AvilesStevo2013,DK,To, Todor99,TU}
and the references therein).  For instance, let $K$  be a  separable compact subset of Baire class 1 functions on a Polish space (also called a (separable) Rosenthal compactum). Let $(f_n)_n$ be a dense subset of $K$ which accumulates to a function $f$ in $K$. Consider the ideal $I_f$ on $\N$  given by $A\in I_f$, if $f$ belongs to  the closure of $\{f_n: \: n\in A\}$.  From the work of
 Krawczyk \cite{Krawczyk92} and Todor\v{c}evi\'{c} \cite{Todor99,Todor2010} we know that several interesting
 topological properties of  $K$ are equivalent to structural or combinatorial properties of the ideal $I_f$.   For example, Todor\v{c}evi\'{c} \cite[Corollary 7.52]{Todor99} showed that $I_f$ is a selective ideal  (see the definition in \S \ref{preli}),  which is the combinatorial counterpart of the fact that $K$ is bisequential.
  In addition, Todor\v{c}evi\'{c} and Avil\'{e}s in \cite{AvilesStevo2015} gave a description of some special classes of separable Rosenthal compacta  in terms of  a combinatorial structure called  strong $n$-gaps \cite{AvilesStevo2012}, which is a generalization of pairs of the form $(I, I^\perp)$.  In a more general setting,   Krawczyk \cite{Krawczyk92} and Todor\v{c}evi\'{c}  \cite{Todor99,Todor2010} have shown that    if $I$ is a selective analytic ideal
not countably generated, then $I^\perp$ is a complete co-analytic
set (see also  \cite{DK}). Another example of the type of results that  motivates our work is  a theorem of  Todor\v{c}evi\'{c} that says that  an analytic
$p$-ideal $I$ is countably generated iff $I^\perp$ is Borel \cite[Theorem 7]{To}.

An ideal $I$ is  said to be {\em Fr\'{e}chet} if $I=I^{\perp\perp}$.
We will recall later the connection of this definition with the
more familiar notion of a Fr\'{e}chet topological  space. Mathias
\cite{Mathias77} showed that every selective analytic ideal is
Fr\'{e}chet (see also \cite[Theorem 7.53]{Todor2010}). But the
converse is not true. So our initial motivation was to study
Fr\'{e}chet ideals such that $I$ and $I^\perp$ are both analytic. Since sets that are analytic and co-analytic are Borel, we will study Fr\'{e}chet Borel ideals with Borel orthogonal, that is, Borel gaps \cite{To}. For example,  a countably generated ideal
satisfies these conditions. The next observation is that a
countable direct sum $\oplus_n I_n$ of Fr\'{e}chet Borel ideals is
also Borel and Fr\'{e}chet, moreover, its orthogonal $(\oplus_n
I_n)^\perp$ is also Borel. Our first result is the following

\medskip \noindent {\bf Theorem I.} {\em The smallest collection
$\base$ of ideals on $\N$ containing the ideal of finite sets and
closed under countable direct sums and the operation of taking
orthogonal has exactly $\aleph_1$ non isomorphic Fr\'{e}chet ideals.
Moreover, all ideals in $\base$ have complexity
$F_{\sigma\delta}$.}

\medskip

We will define a sequence $P_\alpha$ for $\alpha<\omega_1$ of
ideals such that an ideal belongs to $\base$ iff it is isomorphic
to one of the following: $P_\alpha$, $P_\alpha^\perp$ or
$P_\alpha\oplus P_\alpha^\perp$. To give an example, let's denote by
$\FIN$ the ideal of finite sets and by $I^\omega$ the countable
direct sum of copies of an ideal $I$ (note this is not the usual Fubini power of ideals). Then $P_0={\cal P}(\N)$,
$P_1= \FIN^\omega$ and $P_2=((\FIN^{\omega})^{\perp})^{\omega}$. $P_2$ is
the simplest example of a Fr\'{e}chet ideal $J$ such that $J$ and
$J^\perp$ are non isomorphic Borel ideals and neither one is
countably generated. We should also mention that the ideals in
$\base$ are the only examples we know of Fr\'{e}chet Borel ideals
with Borel orthogonal.

The problem of constructing Fr\'{e}chet ideals with special
properties or uncountable families of pairwise non isomorphic
Fr\'{e}chet ideals on $\N$ has been addressed  in the literature
\cite{GU2009,Garcia-Rivera2013,si98,Simon2008,TU}. Those
constructions usually make use  of almost disjoint families of
size the continuum and, in general, the filters produced are not
definable or at least non Borel. For instance, a typical Fr\'{e}chet
ideal is given by ${\mathcal A}^\perp$ where $\cal A$ is an almost
disjoint family of infinite subsets of $\N$. When $\cal A$ is
analytic,  Mathias \cite{Mathias77} showed that the ideal
generated by ${\cal A}$ is selective and by the results of
Todor\v{c}evi\'{c} \cite[Theorem 7.53]{Todor2010} and  Krawczyk
\cite{Krawczyk92}, $\cal A^\perp$ is Borel only if $\cal A$ is
countable (see  also \cite{DK}). In contrast with this, the
collection $\base$ consists of Fr\'{e}chet Borel ideals.

It is known that a Fr\'{e}chet analytic  ideal fails to be selective
iff some restriction of it is isomorphic to $P_1=\FIN^\omega$
\cite[Corollary 4.5]{TU}. Moreover, if $I$ is a selective ideal
and $I\!\!\upharpoonright \!A$ is isomorphic to an ideal in
$\base$, then $I\!\!\upharpoonright \!A$ is isomorphic to either
${\cal P}(\N)$, $\FIN$ or  $(\FIN^{\omega})^\perp$ (these three
ideals and their direct finite sums are the only selective ideals
in $\base$). Thus it seems natural to investigate, for a given
Fr\'{e}chet ideal $I$, which ideals in $\base$ appear as a
restriction of $I$.  To illustrate further this idea we need to
recall the definition of two well known ideals.

Let's denote by $I_{wf}$ the ideal on $\N^{<\omega}$ (the collection of
finite sequences of integers) generated by the well founded trees
on $\N$. $I_{wf}$ is a complete co-analytic subset of
$2^{\N^{<\omega}}.$ In \cite{DK,Krawczyk92} it was made clear the important role played by
$I_{wf}$ in the study of the descriptive complexity of orthogonal
families. We will show the following.

\medskip

\noindent {\bf Theorem II.} {\em Every member of $\mathcal{B}$ is
isomorphic to some restriction of $I_{wf}$.}

\medskip

In view of the last result, it is natural to investigate which restriction of an ideal belongs to $\base$. The following theorem seems to indicate that the collection $\base$ could play a critical role for  studying the complexity of Fr\'{e}chet  ideals.

\medskip

\noindent {\bf Theorem III.} {\em Let $I$ be either $I_{wf}$,   $WO(\Q)$ or $J^\perp$ with $J$ an analytic selective ideal.   The
following are equivalent.
\begin{itemize}
\item[(i)] $I\!\!\upharpoonright \!A$ is isomorphic to an
ideal in $\mathcal{B}$.
\item[(ii)] $I\!\!\upharpoonright
\!A$ is Borel.
\end{itemize}
}

%\medskip

There is an analogy between the collection $\base$ and the
hierarchy of countable scattered linear orders given by the
classical Hausdorff's theorem. In fact,  we show that $A\su \Q$ is scattered iff $WO(\Q)\!\!\upharpoonright \!A$ is isomorphic to an
ideal in $\mathcal{B}$.  However, the ideals $WO(\Q)$ and
$I_{wf}$ are structurally very different, since $WO(\Q)$ is
isomorphic to its orthogonal and this is not true for $I_{wf}$.

\medskip

We would like to recall the reason for calling Fr\'{e}chet
an ideal $I$ such that $I=I^{\perp\perp}$. To each ideal $I$ on
$\N$ we associate a topology on $X=\N\cup\{\infty\}$ where each
$n\in\N$ is isolated and  the neighborhoods of $\infty$ are all sets of
the form $V\cup\{\infty\}$ with $V$ in the dual filter of $I$.
A topological space $Z$ is said to be Fr\'{e}chet, if
whenever $A\su Z$ and $z\in \overline{A}$, there is a sequence
$(z_n)_n$ in $A$ converging to $z$. It is easy to verify that $X$,
with the topology defined above, is Fr\'{e}chet iff $I$ is
Fr\'{e}chet: Just notice that a sequence $S\su \N$ converges to
$\infty$ iff $S\in I^\perp$.

When both the topology of a space $X$ (in the example above, given
by the ideal $I$) and the convergence relation (given by
$I^\perp$) are Borel, we could say that the space $X$  is
definable in a strong sense. The result of Krawczyk about
Rosenthal compacta \cite{Krawczyk92} says that this is not the
case when the compactum is not first countable. Nevertheless, Debs
\cite{Debs87,Debs2009} has shown that there is a Borel set that
codes the convergence relation in a Rosenthal compactum. In the
last section we shall see how the ideals in $\base$ can be used to
construct a family of size $\aleph_1$
of pairwise non homeomorphic countable sequential spaces such that  both the
topology and the convergence relation are Borel and yet all of them
have sequential order $\omega_1$. The spaces we construct are homeomorphic
to some subspaces of $C_p(\N^\N)$. This will answer a question posed in \cite{TU}.

\bigskip

Finally, we would like to state some  questions left open in this paper. The main issue  is whether a Fr\'{e}chet ideal $I$ such that  $I$ and  $I^\perp$ are both  Borel is necessarily isomorphic to an ideal in $\base$, that is, whether
$\base$ is a complete list of Borel gaps.  If that is not the case, are all Borel Fr\'{e}chet ideals with Borel orthogonal $\fsigdel$? Of bounded complexity?  Are there continuum many non-isomorphic such ideals?

\section{Preliminaries and notation}\label{section2}
\label{preli}

The set $\N^{<\omega}$ will denote the set of finite sequences of
integers and $|s|$ denotes the length of the sequence $s\in
\N^{<\omega}$. The set $\N^{\omega}$ will denote the set of
infinite sequences of integers. We will say that a sequence $s\in
\N^{<\omega}$ extends a sequence $t\in \N^{<\omega}$, denote
$t\preceq s$,  if for all $i<|t|$ we have that $s(i)=t(i)$.   A
{\em tree} is a collection of sequences downward closed under
$\preceq$. Given a finite sequence $t$ and an integer $n$ we
denote the sequence $\langle t(0),\dots,t(|t|-1),n \rangle$ by $t
^{\frown} n$. If $A$ is a subset of $\N^{<\omega}$, $\langle A
\rangle $ denotes the tree generated $A$, whereas $A_{t}$ denotes
the set $\{s\in A : t\preceq s\}$  and
$\mathcal{N}_{t}=\{s\in \N^{<\omega}: t\preceq s\}$ for each $t\in \N^{<\omega}.$

An {\em ideal} over a  set $X$ is a family $I$ of subsets of $X$
that contains the empty set, it is closed under taking subsets and
finite unions. For convenience, we will allow the trivial ideal
${\cal P}(X)$ (i.e. when $X\in I$).  We will suppose that an ideal over $X$ contains every finite subset of $X$.
$\FIN$ denotes  the ideal of finite subsets of $\N$ and $\N^{[\infty]}$ the family of infinite
subsets of $\N$. An
ideal  $I$ is {\em selective} \cite{Mathias77} if for all decreasing
sequence of sets $A_n\not\in I$ for $n\in\N$,
 there is $B\not\in I$ such that $B\setminus
\{0,\cdots, n-1\}\subseteq A_n$ for all $n\in B$.

An ideal $I$ on $X$ is isomorphic to an ideal $J$ on $Y$ if there
is a bijection $f: X \rightarrow Y$ such that $A\in I$ iff
$f[A]\in J$; this will be denoted by $I\cong J$.  If $f$ is just
an injection, we will write  $I\hookrightarrow J$ and say that $J$
has a copy of $I$ and  we will write
$I\stackrel{f}{\hookrightarrow} J$ to denote that $f$ is an
isomorphic embedding that witness $I\hookrightarrow J$. If $K$ is
a subset of $X$, the restriction of $I$ to $K$, denoted by $I\up
K$, is the ideal on $K$ consisting of all the subsets of $K$
belonging to $I$.

A subset of a Polish space is called  {\em analytic} if it is
a continuous image of a Borel subset of a Polish space. It is
called {\em co-analytic} when its complement is analytic. By the
usual identification of subsets with characteristic functions, we
can identify an ideal on $\N$ with a subset of the Cantor cube
$\cantor$ and thus it makes sense to say that an ideal is Borel,
analytic, co-analytic, etc.

Let $\cal A$ be a collection of subsets of $X$, the orthogonal
$\cal A^\perp$ of $\cal A$ was defined in the introduction. This
terminology is taken from \cite{To}.  Two families of sets $\cal
A$ and $\cal B$ are orthogonal if ${\cal A}\subseteq {\cal
B}^\perp$. It is easy to verify that
$I^{\perp}=I^{\perp\perp\perp}$ and $I\subseteq I^{\perp\perp}$.
An ideal $I$ has the {\em Fr\'{e}chet property} or just is a {\em
Fr\'{e}chet} ideal, if $I=I^{\perp\perp}$.

Let $\{K_n:\; n\in F\}$ be a partition of $X$, where $F\su \N$.
For $n\in F$, let $I_n$ be an ideal on $K_n$. The direct sum,
denoted by $\bigoplus\limits_{n\in F} I_n$, is defined  by
\[
A\in \bigoplus_{n\in F} I_n \Leftrightarrow (\forall n\in F)(A\cap K_n \in I_n).
\]
In general, given a sequence of ideals $I_n$ over a countable set
$X_n$, we define $\oplus_n I_n$ by taking a partition $\{K_n:\;
n\in \N\}$ of $\N$ and an isomorphic  copy $I_n'$ of $I_n$ on
$K_n$ and let $\oplus_n I_n$ be $\oplus_n I_n'$. It should be
clear that  $\oplus_n I_n$ is,  up to isomorphism, independent of
the partition and the copy used. If all $I_n$ are equal to $I$ we
will write $I^\omega$ instead of $\oplus_n I_n$.

For example, if we sum infinite many times the ideal $\FIN$ we get
$\FIN^\omega$ a well known ideal, sometimes denoted by
$\emptyset\times \FIN$. Its orthogonal $\FIN^{\omega\perp}$,
 sometimes is denoted by $\FIN\times
\emptyset$. Those ideals  play a crucial role for the general
study of analytic ideals (\cite{Farah2000, Solecki1999}).
Moreover, the topological space associate to $\FIN^\omega$ (as
explained in the introduction) is the sequential fan, which is the
prototypical example of a non first countable  Fr\'{e}chet space.

We present some basic facts about $\oplus$ that will be used in the sequel.

\begin{lem}
\label{properties_sumas_directas} Let $I$, $J$ and $K$ be ideals.
\begin{enumerate}
\item[(i)] $I\oplus J = J\oplus I$.

\item[(ii)] $(I\oplus J)\oplus K\cong I\oplus (J\oplus K)$.

\item[(iii)] Parts (i) and (ii) also hold for infinite sums.

\item[(iv)] $(I\oplus J)^\perp\cong I^\perp\oplus J^\perp$.

\end{enumerate}
\end{lem}

\begin{lem}
\label{lemma1}
Let $\{K_n:\; n\in\N\}$ be a partition of a countable set $X$ with each $K_n$ infinite.  Let $I_n$ be an ideal on $K_n$ for each $n\in \N$.

\begin{enumerate}

\item[(i)] If each  $I_n$ is Borel, then  $\oplus_n
I_n$ is also Borel.

\item[(ii)]
\[
\begin{array}{lcl}
              A\in (\oplus_n I_n)^{\perp} & \Leftrightarrow & (\exists k\in\N)(A\su \cup_{i\leq k} K_i \textrm{ and } (\forall i\leq k)\; A\cap K_i\in I_{i}^{\perp}) \\
                                          & \Leftrightarrow & [(\forall i\in\N)(A\cap K_i\in I^\perp_i)] \textrm{ and  } [(\exists n\in\N)(\forall i>n) (A\cap K_i=\emptyset)].
\end{array}
\]

%$A\in (\oplus_n I_n)^{\perp} \Leftrightarrow (\exists k\in\N)(A\su \cup_{i\leq k} K_i \textrm{ and } (\forall i\leq k)\; A\cap K_i\in I_{i}^{\perp})$.

\item[(iii)] If  $I_{n}^{\perp}$  is Borel for each $n$, then
$(\oplus_n I_n)^{\perp}$ is also Borel.

\item[(iv)] If $I_n$ is $F_{\sigma\delta}$ for all $n$, then so are $\oplus_n I_n$ and $(\oplus_n I_n)^{\perp}$,

\end{enumerate}

\end{lem}

%\begin{lem}\label{complexityofideals} \textrm{ \; }
%\begin{enumerate}
%\item[(i)] If $I_n$ is $F_{\sigma\delta}$ for all $n$, then so is $\oplus_n I_n$.
%\item[(ii)] If $\oplus_n I_n$ is $F_{\sigma\delta}$, then so is $(\oplus_n I_n)^{\perp}$
%\end{enumerate}
%\end{lem}

\begin{lem}
\label{frechet}
If $I_n$ is Fr\'{e}chet for all $n$, then $\oplus_n I_n$ is Fr\'{e}chet.
\end{lem}

\proof
Let $(K_n)_n$ be the partition of $\N$ that defines $\oplus_n I_n$.
Take an infinite set $A\su \N$ that is not in $\oplus_n I_n$. Then,
there is $n_0\in\N$ such that $A\cap K_{n_0}\notin I_{n_0}$.
Since $I_{n_0}$ is Fr\'{e}chet, there is an infinite set
$B\su A\cap K_{n_0}$ belonging to $I_{n_0}^{\perp}$. It is clear that  $B$ is also in
$(\oplus_n I_n)^{\perp}$.
\fin

\bigskip

Now we define two ideals that play an important role in our
results. Consider the ideal $I_{wf}$ generated by the well founded
trees on $\N$. We will call a set $A\su \N^{<\omega}$ well founded
if it belongs to $I_{wf}$, that is to say, if there is a
wellfounded tree $T$ such that $A\su T$. Obviously, this is
equivalent to say that the tree generated by $A$ is well founded.
The orthogonal of $I_{wf}$ is the ideal $I_{d}$ generated by the
finitely branching trees on $\N$, or equivalently, $I_d$
consists of  sets which are dominated by a branch:
\[
A\in I_d\;\Leftrightarrow \; \exists \alpha\in \N^\omega \forall
s\in A\forall i<|s| (s(i)\leq \alpha(i))
\]

The ideal $I_{wf}$ is a complete co-analytic set \cite{DK} while
the ideal $I_{d}$ is easily seen to be $F_{\sigma\delta}$.

\section{The family $\mathcal{B}$}\label{section3}

One of the main purposes of this work is to study the smallest
collection $\mathcal B$ of ideals on $\N$ containing $\FIN$  and closed
under the operation of taking countable sums and orthogonal.  In
this section we give a precise characterization of the ideals
belonging to $\mathcal B$ and in particular we show that $\mathcal
B$ has exactly $\aleph_1$ non isomorphic elements. All ideals in
$\mathcal{B}$ are Borel and Fr\'{e}chet; moreover, we will see later
that the  members of $\mathcal{B}$ have Borel complexity at most
$F_{\sigma\delta}$.

The members of $\base$ are, by definition, ideals on $\N$, but we will regard $\base$ as if it were closed under isomorphism, thus when we say that an ideal $I$ over a countable set $X$ belongs to $\base$, we actually mean that $I$ is isomorphic to an ideal in $\base$. To state our results we define by recursion a sequence of ideals $P_\alpha$ and $Q_\alpha$ for $\alpha<\omega_1$.
For every limit ordinal $\alpha<\omega_1$ we fix an increasing
sequence $(\upsilon_{n}^{\alpha})_n$ of ordinals  such that
$\sup_n(\upsilon_{n}^{\alpha})=\alpha$.

\begin{itemize}
\item[(i)] $P_0=\mathcal{P}(\mathbb{N})$ and $Q_0=P_{0}^{\perp}=\FIN$.
\item[(ii)] $P_{\alpha+1}=(P_{\alpha}^{\perp})^\omega$.
\item[(iii)]  $P_{\alpha}=\oplus_{n}P_{\upsilon_{n}^{\alpha}}^\perp$,  for $\alpha<\omega_1$ a limit ordinal.
\item[(iv)]  $Q_\alpha=P_\alpha^\perp$ for every $\alpha<\omega_1$.
\end{itemize}

The following result follows immediately from the definition of $P_\alpha$.

\begin{lem}
\label{particion-P}
For each $\alpha<\omega_1$ there is a partition of $(K_n)_n$ of $\N$ such that, letting  $I_n=P_{\alpha}\upharpoonright K_n$, we have:
\begin{itemize}

\item[(i)] If $\alpha=\beta+1$, then $P_{\alpha}= \oplus_n I_n$ and $I_n\cong Q_{\beta}$ for all $n$.
\item[(ii)] If $\alpha$ is a limit ordinal, then $P_\alpha=\oplus_n  I_n$, and $I_n\cong Q_{\upsilon_{n}^{\alpha}}$ for all $n$.
\end{itemize}
\fin
\end{lem}

It should be clear that the definition of $P_\alpha$, $\alpha<\omega_1$
is, up to isomorphism, independent of the partition used. On the
other hand, we will show below that for $\alpha$ limit it is also
independent of the sequence $(\upsilon_n^\alpha)_n$.

For instance, a standard copy of $P_1=\FIN^\omega$ is defined on
$\N^2$ as follows: $A\in P_1$ iff $\{m\in\N:\; (n,m)\in A\}$ is
finite for all $n\in \N$. Therefore $A\in P_1^\perp$ iff there is
$n$ such that $A\su \cup_{k=0}^n \{k\}\times\N$.

Our first result about  $\mathcal{B}$ is the following.

\begin{thr}
\label{theorem1}  Every ideal in $\mathcal{B}$ is isomorphic to
either $P_{\alpha}$, $Q_{\alpha}$ or $P_{\alpha}\oplus
Q_{\alpha}$ for some $\alpha<\omega_1$.
\end{thr}

The proof consists in showing that the collection of all ideals isomorphic to either $P_\alpha$, $Q_\alpha$ or $P_{\alpha}\oplus Q_{\alpha}$ (for some $\alpha<\omega_1$) is closed under finite or countable direct sums and orthogonal.

\bigskip

\begin{lem}
\label{idempotencia}
\begin{itemize}

\item[(i)] $P_\alpha\oplus P_\beta\cong P_\alpha$, if $\beta\leq
\alpha$.
\item[(ii)] $Q_\alpha\oplus Q_\beta\cong Q_\alpha$, if $\beta\leq
\alpha$.
\item[(iii)] $P_\alpha\oplus Q_\beta\cong P_\alpha$, if $\beta<
\alpha$.
\item[(iv)] $Q_\alpha\oplus P_\beta\cong Q_\alpha$, if $\beta<
\alpha$.
\end{itemize}
\end{lem}

\proof By passing to the orthogonal and using Lemma \ref{properties_sumas_directas} we get that  (i) and (ii) are equivalent. The same occurs with (iii) and (iv). The rest of the proof is by induction on $\alpha$. The result is obvious for $\alpha=0$. It is easy to see that $P_1\oplus P_1\cong P_1 \cong P_1\oplus Q_0$. Now we show that $P_1\oplus P_0\cong P_1$. Consider the standard copy of $P_1$ defined above and the function $f:\N\cup \N^2\rightarrow \N^2$ given by $f(n)= (n,0)$, $f(n,m)= (n,m+1)$. It is left to the reader to check that $f$ is an isomorphism between $P_0\oplus P_1$ and
$P_1$.

Suppose the result holds for all ordinals smaller than $\alpha$.

(i). Let $\beta<\alpha$. We show that $P_\alpha\oplus P_\beta\cong P_\alpha$. If $\alpha$ is a limit ordinal, then $P_{\alpha}=\oplus_{n} Q_{\upsilon_{n}^{\alpha}}$. Therefore,
there must be $n_0\in\N$ such that
$\beta<\upsilon_{n_0}^{\alpha}<\alpha$. By the induction
hypothesis  $Q_{\upsilon_{n_0}^{\alpha}} \oplus P_\beta
\cong Q_{\upsilon_{n_0}^{\alpha}}$. Hence

\begin{equation}
\label{elminar-uno}
(\oplus_n Q_{\upsilon_{n}^{\alpha}})\oplus
P_\beta = \oplus_{n\neq n_o} Q_{\upsilon_{n}^{\alpha}}\oplus
(Q_{\upsilon_{n_0}^{\alpha}} \oplus P_\beta) \cong (\oplus_n
Q_{\upsilon_{n}^{\alpha}})
\end{equation}
%\]
Suppose now that $\alpha=\mu + 1$. Then
$P_{\alpha}=Q_{\mu}^{\omega}$. There are two cases to consider.  If $\beta<\mu$, by the induction hypothesis, we have that
$Q_{\mu}\oplus P_\beta\cong Q_{\mu}$ and, as in \eqref{elminar-uno},
we get that $Q_\mu^\omega\oplus P_\beta\cong Q_\mu^\omega$. Now,
if $\beta=\mu$, we have that $P_\beta=\oplus_n Q_{\xi_n}$, where
$\xi_n<\mu$ (no matter if $\mu$ is limit or not). By the inductive
hypothesis, $Q_{\mu}\oplus Q_{\xi_n}\cong Q_\mu$ and therefore
\begin{equation}
\label{elminar-uno-b}
Q_\mu^\omega \oplus P_\beta = Q_{\mu}^{\omega}\oplus \oplus_n Q_{\xi_n}\cong
\oplus_n (Q_{\mu}\oplus Q_{\xi_n}) {\cong} Q_{\mu}^{\omega}
\end{equation}
Thus, we have shown that $P_\alpha\oplus P_\beta\cong P_\alpha$
for $\beta<\alpha$.

Now we show $P_\alpha\oplus P_\alpha\cong
P_\alpha$. The argument is similar. If $\alpha$ is limit, we argue as in \eqref{elminar-uno-b}. And for $\alpha=\mu+1$,
we use that $J^\perp\oplus J^\perp\cong (J\oplus J)^\perp$.

(iii). The proof is entirely similar and is left to the reader.

\fin

\begin{lem}
\label{sumas-directas}
Let $(\xi)_n$ be a sequence of countable ordinals and $\alpha=\sup_n \xi_n$ (including the case $\xi_n=\alpha$ for some $n$).
\begin{enumerate}
\item[(i)] $(P_\beta)^\omega \cong P_\beta$ for all $\beta<\omega_1$.

\item[(ii)] Suppose $\xi_n<\alpha$  and $I_n\in\{ Q_{\xi_n},  P_{\xi_n},  Q_{\xi_n}\oplus  P_{\xi_n}\}$ for all $n$. Then $P_\alpha\cong \oplus_n I_n$.

\item[(iii)] $\oplus_n P_{\xi_n}\cong P_\alpha$.

\item[(iv)] $\oplus_n Q_{\xi_n}$ is equivalent to either $P_{\alpha+1}$, $P_\alpha$ or $P_\alpha\oplus Q_\alpha$.

\end{enumerate}

\end{lem}

\proof (i) Suppose first that $\beta=\mu+1$ is a successor ordinal. By Lemma \ref{idempotencia} (ii) we have $Q_\mu\oplus Q_\mu\cong Q_\mu$. Therefore,  $(P_{\mu+1})^\omega =((Q_\mu)^\omega)^\omega\cong (Q_\mu\oplus Q_\mu)^\omega\cong (Q_\mu)^\omega =P_{\mu+1}$.

Now suppose that $\beta$ is a limit ordinal.  By definition, $P_\beta=\oplus_n Q_{\upsilon_{n}^{\beta}}$ and $\beta=\sup_n \upsilon_{n}^{\beta}$.
 Pick an increasing sequence $(n_k)_k$ of integers such that ${\upsilon_{k}^{\beta}}<{\upsilon_{n_k}^{\beta}}$ for all $k$. By Lemma \ref{idempotencia}, $P_{\upsilon_{k}^{\beta}+1}\oplus Q_{\upsilon_{n_k}^{\beta}}\cong Q_{\upsilon_{n_k}^{\beta}}$ for all $k$. Let $A=\N\setminus\{n_k:\; k\in\N\}$.
It is easy to verify the following
\[
(P_\beta)^\omega\cong \oplus (Q_{\upsilon_{n}^{\beta}})^\omega\;\;\;\mbox{and}\;\;\; (P_\beta)^\omega\cong P_\beta\oplus (P_\beta)^\omega.
\]
From this and the fact that, by definition, $(Q_{\upsilon_{n}^{\beta}})^\omega= P_{\upsilon_{n}^{\beta}+1}$ for all $n$, we have
\[
\begin{array}{lcl}
(P_\beta)^\omega & \cong &(\oplus_n Q_{\upsilon_{n}^{\beta}}) \oplus (\oplus_n P_{\upsilon_{n}^{\beta}+1})\\
                              &\cong  & (\oplus_{n\in A} Q_{\upsilon_{n}^{\beta}})\oplus (\oplus_k P_{\upsilon_{k}^{\beta}+1}\oplus Q_{\upsilon_{n_k}^{\beta}})\\
                              &\cong & (\oplus_{n\in A} Q_{\upsilon_{n}^{\beta}})\oplus (\oplus_kQ_{\upsilon_{n_k}^{\beta}})\\
                              &\cong& P_\beta.
 \end{array}
\]

(ii)  First we show that we can assume that $(\xi_n)_n$ is strictly increasing. Let $n_0=0$ and $n_{k+1}=\min\{m:\; \xi_{n_k}<\xi_m\}$. Notice that $(\xi_{n_k})_k$ is strictly increasing. By Lemma \ref{idempotencia}, we have that for all $k$
\[
\oplus\{I_m:\; n_{k-1}<m\leq n_k\}\cong I_{n_k} \;\; \mbox{(where $n_{-1}=-1$).}
\]
By the same argument, we have that for any strictly increasing sequence $(m_k)_k$ the following holds
\begin{equation}
\label{absorcion}
\oplus_n I_n\cong \oplus_k I_{m_k}
\end{equation}
Fix two increasing sequences $(n_k)_k$ and $(m_k)_k$ such that $\xi_{n_k}\leq {\upsilon_{m_k}^{\alpha}}<\xi_{n_{k+1}}$ for all $k$. From \eqref{absorcion} and Lemma \ref{idempotencia},  we have
\[
P_\alpha=\oplus_m Q_{\upsilon_{m}^{\alpha}} \cong \oplus_k Q_{\upsilon_{m_k}^{\alpha}} \cong \oplus_k (Q_{\upsilon_{m_{k+1}}^{\alpha}}\oplus I_{n_{k}})\cong \oplus_k (Q_{\upsilon_{m_k}^{\alpha}}\oplus I_{n_{k+1}})\cong\oplus_k I_{n_k}\cong \oplus I_n.
\]
This finishes the proof of (ii).

\medskip

Before proving (iii) and (iv) we introduce some notation. Let $A=\{n\in\N:\; \xi_n=\alpha\}$, $B=\N\setminus A$ and  $\beta=\sup_n \{\xi_n:\; n\in B\}$.

\medskip

(iii) By induction on $\alpha$. If $\alpha=0$, the result holds trivially since $P_0\cong (P_0)^\omega$. Suppose it holds for all ordinals smaller than $\alpha$. We consider two cases. If $A$ is empty, then $\alpha$ is a limit ordinal and the result follows from part (ii). Suppose $A$ is not empty. Then
$$\oplus_n P_{\xi_{n}}\cong (\oplus_{n\in A} P_{\xi_n})\oplus (\oplus_{n\in B} P_{\xi_n})
$$
(if $B$ is empty, we do not include the second summand).  If $A$ is infinite, then $\oplus_{n\in A} P_{\xi_n}\cong (P_\alpha)^\omega \cong P_\alpha $ by part (i). From this, part (ii) and Lemma \ref{idempotencia} we conclude $\oplus_n P_{\xi_{n}}\cong P_\alpha$.   When $A$ is finite, the argument is similar.

\medskip

(iv) By induction on $\alpha$. Let $I= \oplus_n Q_{\xi_n}$. If $\alpha=0$, then $I$ is  $P_1$. If $A$ is empty, then  $I\cong P_\alpha$ by part (ii). Suppose $A$ is non empty. Then we use an argument analogous to that used in the proof of (iii). In fact,   if $A$ is finite, then $I\cong P_\alpha\oplus Q_\alpha$. If $A$ is infinite, then $I\cong P_{\alpha+1}$.
\fin

\bigskip

\bigskip

\noindent {\bf Proof of Theorem  \ref{theorem1}:}
From Lemmas \ref{idempotencia} and \ref{sumas-directas} we have that the collection $\mathcal C$ of all ideals isomorphic to either $P_\alpha$, $Q_\alpha$ or $P_\alpha\oplus Q_\alpha$, for some countable ordinal $\alpha$, is closed under finite or countable direct sums. Since $(P_\alpha)^\perp= Q_\alpha$, $(Q_\alpha)^\perp = P_\alpha$ and $(P_\alpha\oplus Q_\alpha)^\perp\cong P_\alpha^\perp\oplus Q_\alpha^\perp$, then $\mathcal C$ is also closed under orthogonal. Therefore $\mathcal{C}=\mathcal{B}$.
\fin

\begin{thr}\label{complxofB}
All  members of $\base$ are Fr\'echet of  Borel complexity at most $F_{\sigma\delta}$.
\end{thr}

\begin{proof}
From  Theorem \ref{theorem1} and Lemma \ref{frechet}, every member of $\base$ is a Fr\'{e}chet ideal. Notice that  $\FIN$ is $F_{\sigma\delta}$ (in fact $F_{\sigma}$) and  $\oplus I_n$ is $F_{\sigma\delta}$ if each $I_n$ is  $F_{\sigma\delta}$. The rest follows from  Theorem \ref{theorem1} and  Lemma~\ref{lemma1}.
\fin
\end{proof}

Most of ideals in $\base$ are complete $\fsigdel$. Clearly $P_0$ and $Q_0$ are $F_\sigma$. It is well known that
$P_1=\FIN^\omega$ is $\fsigdel$-complete \cite[pag.
179]{Kechris94} and $Q_1=(\FIN^{\omega})^{\perp}$ is $F_{\sigma}$ using
 Lemma~\ref{lemma1}. Now we take $J\in \mathcal{B}\setminus \{P_0, Q_0, P_0\oplus Q_0, P_1, Q_1, P_1\oplus Q_1\}$,
from Theorem \ref{theorem1} we have that $ \FIN^\omega \oplus J
\cong J $. Hence, $\FIN^\omega \hookrightarrow J$ and therefore
$\FIN^\omega \leq_{W} J$ (where $\leq_W$ is the Wadge reducibility
relation \cite{Kechris94}). Thus $J$ is  $\fsigdel$-complete.

From Theorem \ref{theorem1}, we know there are, up to isomorphism,
at most $\aleph_1$ ideals in $\mathcal B$. Now we will show that
the ideals $P_\alpha$, $Q_\alpha$ and $P_\alpha\oplus Q_\alpha$
are all non isomorphic. This is an inductive proof which  will be  split in several lemmas.

\begin{lem}
\label{restricciones}
 $\base$ is closed under restriction. Moreover,  let $I\in\{ Q_{\alpha},  P_{\alpha},  Q_{\alpha}\oplus  P_{\alpha}\}$,  $\alpha<\omega_1$ and $K\su \N$ infinite.
Then $I\upharpoonright K$ belongs to $\{ Q_{\xi},  P_{\xi},  Q_{\xi}\oplus  P_{\xi}\}$ for some $\xi\leq \alpha$.
\end{lem}
\proof By induction on $\alpha$. It suffices to show the result
for the ideals $P_\alpha$ and $Q_\alpha$. The result is obvious
for $\alpha=0$. The proof follows  from  Lemma \ref{idempotencia}, Lemma \ref{sumas-directas} and the following
two straightforward facts. (i) If $\{K_n:\; n\in \N\}$ is a partition
of $\N$, $I_n$ an ideal over $K_n$ and $K\su \N$ an  infinite set,  then
\[
(\oplus_n I_n) \upharpoonright K\cong \oplus_n (I_n
\upharpoonright K\cap  K_n).
\]
(ii) If  $I$ is an ideal over $\N$, then $I^\perp\upharpoonright K\cong (I\upharpoonright K)^\perp$.
\fin

\begin{lem}\label{lemmageneral}
(i) Let $\rho<\omega_1$ and $K\su\N$ infinite.
Assume that $Q_\rho\ncong Q_{\xi}\up E$, for every $\xi<\rho$ and
every infinite set $E\su\N$. Then, $Q_\rho \ncong P_\rho \up K$.

(ii) Let $\rho<\omega_1$ and $K\su\N$ infinite.
Assume that $Q_\rho\ncong Q_{\xi}\up E$, for every $\xi<\rho$ and
every infinite set $E\su\N$. Then, $Q_{\rho+1} \ncong P_\rho \up K$.

%iii) Consider an ordinal $\rho<\omega_1$ and an ordinal $\beta< \rho$. Assume that $Q_{\rho} \ncong Q_\xi \up E$,
%for all $\xi<\rho$ and all infinite set $E\su\N$. Then, $Q_{\rho}\ncong P_{\beta}\oplus Q_{\beta}$.

(iii) Let $\rho<\omega_1$. Assume that $Q_{\rho+1} \ncong Q_\rho \up E$,
for all infinite set $E\su\N$. Then, $Q_{\rho+1}\ncong P_{\rho}\oplus Q_{\rho}$.
\end{lem}

\begin{proof}
(i) It is trivially true that $Q_0 \ncong P_0 \up K$. Suppose $\alpha>0$ and let $(\xi_n)_n$ be a sequence of ordinals such that $P_\rho=\oplus_n
Q_{\xi_n}$ (note that every $\xi_n$ is less than $\rho$ regardless  whether
$\rho$ is a limit or a successor ordinal). Let $\{ K_n
: n\in\N \}$ be a partition of $\N$ such that $P_\rho\up K_n = Q_{\xi_n}$ for all $n$ (Lemma \ref{particion-P}). Thus
$P_\rho\up K = \oplus_n Q_{\xi_n}\up (K\cap K_n)$.  Suppose, towards a contradiction, that $f: \N\rightarrow K$ is
an isomorphism witnessing $Q_\rho \cong P_\rho \up K$.

Let $L_n$ be $K\cap K_n$. We will define  sequences of integers
$(p_k)_k$, $(n_k)_k$, and $(l_k)_k$ with the following properties.
\begin{enumerate}
\item[(1)] $(n_k)_k$ and $(l_k)_k$ are increasing,
\item[(2)] $p_k\in K_{n_k}$, for all $k\in \N$
and $f(p_k)\in L_{l_k}$, for all $k\in \N$.
\end{enumerate}

Assume  we have constructed such sequences and  we get the required contradiction. Put $A=\{
p_k : k\in\N\}$. Since $(n_k)_k$ is increasing, by Lemma
\ref{lemma1} (ii), $A\notin Q_{\rho}$. On the other hand, as
$(l_k)_k$ is increasing, we have that $f[A]\cap L_{l_k}=\{f(p_k)\}$ for
all $k\in\N$, and hence $f[A]\in P_\rho\up K$. This contradicts
that $f$ is an isomorphism.

The sequences  $(n_k)_k$,  $(l_k)_k$ and $(p_k)_k$  are defined by recursion. Let
$p_0\in K_0$ and $l_0$ be such that $f(p_0)\in L_{l_0}$.
Put $n_0=0$ and $m_0=\max\{n_0,l_0\}$.
Suppose we have chosen $n_{k}$,
$l_{k}>m_{k-1}$ and $p_k\in K_{n_k}$ such that
\[
f(p_k)\notin \bigcup_{i=0}^{m_{k-1}}L_i  \textrm{ and } f(p_k)\in
L_{l_{k}}.
\]

Let $m_k=\max\{n_k,l_{k}\}$. We claim that

\[
(\exists n>m_{k})(\exists p\in K_{n}) (f(p)\notin \bigcup\limits_{i=0}^{m_k} L_i  ).
\]
Otherwise, we have that
\[
f[\bigcup_{n>m_{k}}K_n]\su \bigcup_{i=0}^{m_k}L_i.
\]
Let $D=\bigcup_{n>m_k}K_n$. By Lemma \ref{sumas-directas},  $P_{\rho}\up D\cong P_\rho$, thus  $Q_{\rho}\up D\cong Q_\rho$. Then, by Lemma \ref{idempotencia},
\[
Q_{\rho}\cong Q_{\rho}\up D\cong P_\rho\up f[D]=  (Q_{\xi_0}\oplus \cdots
\oplus Q_{\xi_{m_k}})\up f[D]\cong Q_{\gamma}\up f[D],
\]
where $\gamma=\max\{ \xi_0,\dots,\xi_{m_k} \}<\rho$. This contradicts the hypothesis. Therefore, the claim is proved.

Finally, we finish the construction of the sequences $(n_k)_k$, $(l_k)_k$.
Let ${n_{k+1}}$ and  $p_{k+1}$ be as in the claim and $l_{k+1}$ be greater than $m_k$ such that $f(p_{k+1})\in L_{l_{k+1}}$.

\medskip

(ii)  and (iii) are proved as  case (i).

%(iii) Let $\{ K_n : n\in\N \}$  be a partition of $\N$
%use in the definition of $Q_{\rho+1}$ and let $\{L_n : n\in\N\}$ be a partition of $\N$
%so that $Q_{\rho}$ is define on $L_0$ and $\{L_n:n\geq 1\}$ is the partition use
%in the definition of $P_\rho$. Notice that $Q_\rho \oplus P_\rho$  is defined on $\{L_n: n\in \N\}$.
%Assuming that $Q_\rho \cong P_\rho \oplus Q_\rho$ through a function $f$, we can find a set $A$ that
%is negative for $Q_{\rho+1}$ whose image is in $P_\rho$. For this, we start by choosing $p_0$ in some $K_{n_0}$ so that
%$f(p_0)\in L_1$. Then we can complete the proof by following the proof
%done in part (i) step by step.
\fin
\end{proof}

\begin{lem}\label{lemma9}
(i) $P_\alpha\not\cong P_\beta\up K$ for all $\beta<\alpha$ and all $K\su \N$ infinite

(ii) $P_\alpha\not\cong Q_\beta\upharpoonright K$ for all $\beta\leq\alpha$ and all $K\su \N$ infinite.

(iii) $P_\alpha\not\cong P_\beta\oplus Q_\beta$ for all $\beta\leq\alpha$.

\end{lem}

\begin{proof} The proof is by induction on $\alpha$. It is easy to check
that the result holds for $\alpha\leq 1$. Suppose that (i), (ii) and (iii) hold
for all $\gamma<\alpha$ and we show it for $\alpha$.

\medskip

\noindent (a) Case $\alpha$ limit.

(i) Let $\beta<\alpha$  and $K\su \N$ infinite.
Suppose, towards a contradiction, that $f$
is an isomorphism witnessing $P_\alpha\cong P_\beta\up K$. Let
$\{K_n:\;n\in \N\}$ be the partition of $\N$ such that
$P_\alpha\up K_n\cong  Q_{\upsilon_{n}^{\alpha}}$  (Lemma \ref{particion-P}). Let $m$ be such that
$\beta<\upsilon_m^{\alpha}$. Then $Q_{{\upsilon_m^{\alpha}}} \cong P_\beta\upharpoonright f[K_m]$.  Since $\beta<\upsilon_m^{\alpha}$,  we get a contradiction using Lemma
\ref{restricciones} and the inductive hypothesis.

For parts (ii) and (iii),  fix $\beta<\alpha$  and $K\su \N$ infinite.
Arguing as in part (i) we conclude that $P_\alpha\not\cong Q_\beta\up K$
for all $\beta<\alpha$ and $P_\alpha\not\cong P_\beta\oplus Q_\beta$.

Now we are going to show (ii) and (iii) for $\alpha = \beta $.

(ii) Suppose $\alpha=\beta$. We have just proved that
$P_{\alpha}\ncong P_\beta \up E$, for all $\beta<\alpha$ and all
$E\su \N$ infinite. After taking orthogonal, we get the
hypothesis of Lemma \ref{lemmageneral} (i), thus
$Q_{\alpha}\ncong P_{\alpha} \up K$. Taking orthogonal again we
get $P_{\alpha}\ncong Q_{\alpha}\up K$.

(iii) Suppose $\alpha=\beta$ and, towards a contradiction, that
$P_{\alpha}\cong P_\alpha \oplus Q_\alpha$. Then $Q_\alpha \cong
P_\alpha\up C$, for some infinite set $C\su \N$. Taking orthogonal
we get $P_\alpha \cong Q_\alpha \up C$. This contradicts what we
just proved in part (ii).

\medskip

\noindent (b) Case  $\alpha=\mu+1$. Since  $P_{\alpha}=(Q_\mu)^\omega$, then there is $K_0$ infinite such that  $P_\alpha\upharpoonright K_0\cong Q_\mu$.

(i) Let $\beta<\alpha$ and $K\su \N$ infinite. Suppose, towards a
contradiction, that $f$ is an isomorphism witnessing
$P_\alpha\cong P_\beta\up K$. Thus $Q_\mu\cong
P_\beta\upharpoonright f[K_0]$ and taking orthogonal we get
$P_\mu\cong Q_\beta\upharpoonright f[K_0]$. But that contradicts
part (ii) of our inductive hypothesis, as $\beta \leq \mu<
\alpha$.

(ii) Let $\beta\leq\alpha$ and $K\su \N$ infinite. First, we
suppose $\beta <\mu<\alpha$. Assume, towards a contradiction, that
$f$ is an isomorphism witnessing that $P_{\mu+1} \cong Q_\beta \up
K$.  As before, we have $Q_\mu\cong Q_\beta \up f[K_0]$ and thus $P_\mu \cong P_\beta \up f[K_0]$, which contradicts  the inductive hypothesis. Now suppose that  $\beta = \mu$. Using part (ii) of Lemma
\ref{lemmageneral} we get that $P_{\mu+1}\ncong Q_\mu \up K$.
Finally, suppose $\beta =\mu+1$. In part (i) we just proved that
$P_{\alpha}\ncong P_\beta \up E$, for all $\beta<\alpha$ and all
$E\su \N$ infinite, which is (after taking orthogonal)  the
hypothesis of Lemma \ref{lemmageneral}(i). Thus, $Q_{\alpha}\ncong
P_{\alpha} \up K$. Taking orthogonal we get $P_{\alpha}\ncong
Q_{\alpha}\up K$.

(iii) Let $\beta\leq\alpha$ and $K\su \N$ infinite. First, we
suppose $\beta<\mu$. Assume, towards a contradiction, that
$P_\alpha\cong P_\beta \oplus Q_\beta$. Let $f$ be a map
witnessing this fact. As before,
$Q_\mu\cong (P_\beta \oplus Q_\beta)\up f[K_0]$. By Lemma
\ref{restricciones}, we conclude that  $Q_\mu$ is isomorphic to
either $P_\gamma$, $Q_\gamma$ or $P_\gamma \oplus Q_\gamma$, for
some $\gamma\leq \beta<\mu$. This contradicts the inductive
hypothesis.

Now  suppose $\beta=\mu$. In part (i) we just proved that
$P_{\mu+1}\ncong P_\mu \up E$, for all $E\su \N$ infinite. After
taking  orthogonal, we get the hypothesis of the Lemma
\ref{lemmageneral} (iii). Hence, $Q_{\mu+1}\ncong P_\mu \oplus
Q_\mu$. Again, we take orthogonal to get $P_{\mu+1}\ncong P_\mu
\oplus Q_\mu$.

Finally, suppose  $\beta=\mu + 1$. Assume, towards a contradiction
that $ Q_\alpha \oplus P_\alpha\cong P_\alpha$ and denote by $g$ a
function witnessing this fact. Let $C$ be an infinite set $C\su
\N$ such that $(P_\alpha\oplus Q_\alpha)\up C\cong Q_\alpha$.
Then, $Q_\alpha\cong P_\alpha\up g[C]$. But in part (ii) we just
proved that this is impossible.

\fin
\end{proof}

From the previous results we immediately get the following
\begin{thr}\label{theorem3}
The family $\mathcal{B}$ has $\aleph_1$  pairwise non isomorphic
ideals.
\end{thr}

\medskip

As we mentioned in the introduction, we do not know whether $\base$ is a complete list of Borel gaps. However,  if $I$ is a Fr\'{e}chet ideal (on $\N$) such that $I$ and $I^{\perp}$ are both Borel,
we can consider the smallest collection $\mathcal{B}(I)$  of ideals on $\N$ containing $I$ and closed
under the operation of taking countable sums, orthogonal, and closed under restrictions. Since ${\mathcal B}(I)$
is closed under restrictions, we have that ${\FIN} \in {\mathcal B}(I)$, whence
$\base\su\mathcal{B}(I)$ (by lemma~\ref{idempotencia}). Thus, $\mathcal{B}(I)$ contains at least $\aleph_1$ pairwise non isomorphic ideals.

%  Hay alguna version del lemma~\ref{idempotencia} para $\mathcal{B}(I)$? $I$ se puede recuperar al restringir cualquier miembro de $\mathcal{B}(I)\B$ (suponiendo que $B$ es diferente de $B(I)$'') a algun conjunto infinito?

\section{Borel restrictions of analytic ideals}
\label{section4}

In this section we give a representation of the ideals in $\base$  as restrictions of $I_{wf}$.
More generally,  given an ideal $I$  it is natural to investigate which restrictions of $I$ belong to $\mathcal{B}$.  We analyze this problem for $I_{wf}$,  $WO(\Q)$ and for the orthogonal of a  selective ideals.
We could summarize our results saying that such restrictions belong to $\base$ exactly when they are Borel.

\subsection{The ideals in $\mathcal{B}$ are restrictions of $I_{wf}$}

%In this section we study the Borel complexity of the elements of
%$\mathcal{B}$. The proof is based in a representation of each
%ideal in $\base$ as a restriction of the $F_{\sigma\delta}$ ideal
%$I_{d}$.

\begin{thr}
\label{theorem4} Every member of $\mathcal{B}$ is isomorphic to
some restriction of $I_{wf}$ and also  to some restriction of
$I_d$. % In particular, every member of $\base$ is $\fsigdel$.
\end{thr}

The proof is based in the following facts. Recall that $\mathcal{N}_{t}=\{s\in \N^{<\omega}: t\preceq s\}$ for each $t\in \N^{<\omega}.$

\begin{lem}\label{lemma2}
\begin{enumerate}
\item[(i)] Let $s\in\N^{<\omega}$ and   $B_n$ be an infinite
subset of $\mathcal{N}_{s^{\smallfrown} n}$  for each $n\in\N$.
Then
\[
I_{wf}\up (\cup_n B_n) \cong \oplus_n I_{wf}\up B_n.
\]
\item[(ii)] Let $\theta\in \mathbb{N}^{\omega}$  and put
$s_n=\theta\up (n-1) ^{\smallfrown} (\theta(n) + 1)$ and  fix an
infinite subset $B_n$ of $\mathcal{N}_{s_n}$ for each $n\in \N$.
Then
\[
I_{wf}\up (\cup_n B_n) \cong (\oplus_{n} I_{wf}^{\perp}\up
B_n)^{\perp}= (\oplus_{n} I_{d} \up B_n)^{\perp}.
\]
The figures below illustrate how we are choosing the $B_{n}$'s in i) and ii).
\end{enumerate}
\end{lem}

\begin{figure}[h] % indico que voy a poner una figura y [h] indica que la posición relativa, tambien puedo usar t = top entre otros.

\hfill
\begin{minipage}[h]{.45\textwidth}
\begin{center}
\begin{pspicture}(-4.2,0.2)(4.2,6.2)%\psgrid

\psline{-*}(0,1)(0,2)
\put(-0.1,0.5){$\emptyset$}
\put(0.2,1.5){$s$}

\psline{-*}(0,2)(-3,4)
\psline(-3,4)(-4,6)
\psline(-3,4)(-2,6)
\put(-3.3,5){$B_0$}
\put(-4,3.8){$s^{\smallfrown}0$}

\psline{-*}(0,2)(0,4)
\psline(0,4)(-1,6)
\psline(0,4)(1,6)
\put(-0.3,5){$B_1$}
\put(-1,3.8){$s^{\smallfrown}1$}

\psline{-*}(0,2)(3,4)
\psline(3,4)(2,6)
\psline(3,4)(4,6)
\put(2.7,5){$B_n$}
\put(0.8,4){$\dots$}
\put(2,4){$\dots$}t
\put(3.3,3.8){$s^{\smallfrown}n$}
\end{pspicture}
\end{center}
\caption{Lemma~\ref{lemma2} i)}
\label{fig1}
\end{minipage}
\hfill \, \,
\hfill \, \,
\hfill \, \,
\begin{minipage}[h]{.45\textwidth}
\begin{center}
\begin{pspicture}(-5.2,0.2)(2,6)%\psgrid
\psline(-1,1)(-5,5)
\put(-1.1,0.5){$\emptyset$}
\put(-5.1,5.1){$\theta$}

\psline{*-*}(-1,1)(0,2)
\psline(0,2)(-1,3)
\psline(0,2)(1,3)
\put(-0.3,2.4){$B_0$}

\psline{*-*}(-3,3)(-2,4)
\psline(-2,4)(-4,6)
\psline(-2,4)(0,6)
\put(-2.2,5.5){$B_n$}
\put(-4.2,2.8){{\tiny $\theta\up (n-1)$}}
\put(-4.9,4){{\tiny $\theta\up (n)$}}
\put(-1.8,4){{\tiny $\theta\up (n-1) ^{\smallfrown} (\theta(n) + 1)$}}

\psline[linestyle=dashed]{*-}(-4,4)(-2,4)

\put(-1.9,2.5){$\ddots$}
\end{pspicture}
\end{center}
\caption{Lemma~\ref{lemma2} ii)}
\label{fig2}
\end{minipage}
\hfill
\end{figure}

\begin{proof} (i) Notice that if $A\in I_{wf}\up (\cup_n B_n)$, then $A\cap
B_n$ is also a well founded set, for all $n\in \N$. Therefore $A
\in \oplus_n I_{wf}\up B_n$. Conversely, if $A\cap B_n \in I_{wf}$
for all $n\in \N$, since $\{ s^{\smallfrown}n : n\in \N \}$ is an
antichain, we have that $A$ is a well founded set. So $A\in
\oplus_n I_{wf}\up B_n$. Thus
\[
A\in I_{wf}\up (\cup_n B_n) \Leftrightarrow A \in \oplus_n I_{wf}
\up B_n.
\]
Therefore, $ I_{wf}\up (\cup_n B_n) \cong \oplus_{n} I_{wf}\up
B_n.$

\medskip

(ii) Take $A\su \cup_n B_n $. If $A$ is well founded, it can only
have a non empty intersection with finitely many $B_n$'s (otherwise,
any tree containing $A$ will have $\theta$ as a branch). So there
is $n_0\in\N$ such that $A\su \bigcup\limits_{i\leq n_0} B_i$.
From this and Lemma \ref{lemma1} we have that $A\in (\oplus_n
I_{wf}^{\perp}\up B_n)^{\perp}$. Conversely, if $A\in (\oplus_n
I_d \up B_n)^{\perp}$, by Lemma \ref{lemma1}, there is $n_0\in \N$
such that $A\su \bigcup\limits_{i\leq n_0} B_i$ and $A\cap B_i \in
I_{d}^{\perp}\up B_i=I_{wf}\up B_i$ for every $i\leq n_0$.
Therefore, $A\in I_{wf}\up (\cup_n B_n)$. Thus
\[
A\in I_{wf}\up (\cup_n B_n) \Leftrightarrow A\in (\oplus_n I_d\up
B_n)^{\perp}.
\]
Hence $I_{wf}\up (\cup_n B_n) \cong  (\oplus_{n} I_d\up
B_n)^{\perp}$.
\fin
\end{proof}

\bigskip

\noindent{\bf Proof of Theorem \ref{theorem4}} From Theorem
\ref{theorem1}, it suffices to show that the ideals $P_\alpha$,
$Q_\alpha$ and $P_\alpha\oplus Q_\alpha$ are isomorphic to a
restriction of $I_{wf}$. Since $I_d$ is the orthogonal of $I_{wf}$
and $\base$ is closed under taking orthogonal, then the result
also holds for $I_d$. The proof will be by transfinite induction
on $\alpha$.

For $\alpha=0$, take infinite sets $A\in I_{wf}$ and $B\in I_{d}$
we have that $I_{wf}\up A \cong \mathcal{P}(\N)$, $I_d\up A\cong
\FIN$, $I_{wf}\up B \cong \FIN$ and, $I_d\up B \cong
\mathcal{P}(\N)$.

Suppose that the result hods for all $\xi<\alpha$. By definition,
$P_{\alpha}=\oplus_n Q_{\upsilon_n}$,  where $\upsilon_n<\alpha$
for all $n\in\N$. Notice that  $I_{wf}\up \mathcal{N}_{\langle n
\rangle}\cong I_{wf}$ for each $n\in\N$. So, by the inductive
hypothesis, for each $n$, there is $B_n\su \mathcal{N}_{\langle n \rangle}$
such that $I_{wf}\up B_n\cong Q_{\upsilon_n}$. From this and Lemma
\ref{lemma2}(i) we conclude
\[
I_{wf}\up(\cup_n B_n)\cong \oplus_n I_{wf}\up B_n\cong \oplus_n
Q_{\upsilon_n}\cong P_{\alpha}.
\]

Now we will show the result for $Q_\alpha$.  Notice that $I_{d}\cong I_d
\up\mathcal{N}_{0^{n+1}\widehat{\;\;} 1}$ for each $n\in \N$. By
the inductive hypothesis, for each $n$, there is $B_n\su
\mathcal{N}_{0^{n+1}\widehat{\;\;}1}$ such that $I_d\up B_n\cong
P_{\upsilon_n}$. Hence, by Lemma \ref{lemma2}(ii), where $\theta$
is the constantly equal to zero sequence, we have that
\[
I_{wf}\up (\cup_n B_n)\cong (\oplus_n I_d\up B_n)^{\perp} \cong
(\oplus_n P_{\upsilon_n})^{\perp}\cong Q_{\alpha}.
\]
Thus we have shown that $P_{\alpha}$ and $Q_{\alpha}$ are
isomorphic to a restriction of $I_{wf}$.

Finally,  since $I_{wf}\cong
I_{wf}\up \mathcal{N}_{\langle 0 \rangle }$ and $I_{wf}\cong
I_{wf}\up \mathcal{N}_{\langle 1 \rangle}$, there are infinite
sets $C\su \mathcal{N}_{\langle 0 \rangle }$ and $D \su
\mathcal{N}_{\langle 1 \rangle }$ such that $I_{wf}\up C\cong
P_{\alpha}$ and $I_{wf}\up D\cong Q_{\alpha}$. Thus
\[
I_{wf}\up (C\cup D) \cong I_{wf}\up C \oplus I_{wf} \up D\cong
P_{\alpha}\oplus Q_{\alpha}.
\]

%The last statement follows from the fact that $I_d$ is
%$F_{\sigma\delta}$.
\fin

\bigskip

\subsection{Borel restrictions of the orthogonal of a selective ideal}

The following result is due to Krawczyk \cite{Krawczyk92},  we will state it   as in the work of  P. Dodos and V. Kanellopoulos \cite{DK}. Let $\mathcal{C}$ be the collection of chains in $\seq$.

\begin{thr}
\label{krawczyk} (Krawczyk \cite{Krawczyk92}) Let $I$ be an analytic selective ideal. Then  either
\begin{itemize}
\item[(i)] $I$ is countably generated, or

\item[(ii)] There exists a one-to-one map $\Psi :\seq\rightarrow \N$ such that
$\mathcal{C} \subseteq \{\Psi^{-1}(A):\; A\in I\}$  and $I_{wf}\subseteq \{\Psi^{-1}(B): \; B\in I^{\perp}\}$.
\end{itemize}
\fin
\end{thr}

\begin{thr}\label{teorema_selectivos}
Let $I$ be analytic selective ideal and $A\subseteq \N$. The following are equivalent

\begin{itemize}
\item[(i)] $I\up A$ is countably generated
\item[(ii)] $I^\perp\up A\in \base$.
\item[(iii)] $I^\perp\up A $ is Borel
\item[(iv)] $I_{wf}\not\hookrightarrow I^\perp\up A$.
\end{itemize}
\end{thr}

\proof An ideal is countably generated iff it is a finite direct sum of ideals belonging to  $\{P_0, Q_0, Q_1\}$. Therefore, it is clear that $(i)\Rightarrow (ii) \Rightarrow (iii)\Rightarrow (iv)$.  The remaining implication follows from \ref{krawczyk}. In fact,  suppose $I\up A$ is not countably generated. Since $I\up A$ is analytic and selective, there is $\Psi :\seq\rightarrow A$ as in the statement of the theorem \ref{krawczyk}. We claim  $I_{wf}\hookrightarrow I^{\perp}\up A$. If $E\in I_{wf}$, then clearly $\Psi(E)\in I^{\perp}$. On the other hand, if $E\not\in I_{wf}$, there is $B\in \mathcal{C}$ such that $B\cap E$ is infinite, as $\Psi(B\cap E)$ is infinite and belongs to $I$, then  $\Psi(E) \not\in I^{\perp}$.
\fin

%A source of analytic selective ideals can be found in those represented by a compact set of Baire class 1 functions (see Theorems 7.48, 7.49 and 7.50 of \cite{Todor2010}).
%From Theorem~\ref{teorema_selectivos} we get the following reformulation of Theorem 6.64 of \cite{Todor2010}. Let $I_{\mathcal{C}}$ denote the ideal on $2^{N}$
%generated by all chains in $(2^{\N}, \preceq)$.
%
%\begin{thr}
%If I is an ideal (not isomorphic to $P_0$) represented by a compact set of Baire class 1 functions, then either
%
%\begin{itemize}
%\item[(i)] $I\up A\in \mathcal{B}$, for some $A\su\N$, or
%\item[(ii)] $I\up A \cong I_{\mathcal{C}}$, for some $A\su\N$.
%\end{itemize}
%\end{thr}
%
%\footnote{\bf \large Paco, ya revisaste que pasa con (ii) del teorema? si es isomorfismo?}

\subsection{Borel restrictions of $I_{wf}$}

In this section, we analyze the Borel restriction of the ideal $I_{wf}$.

\begin{thr}
\label{theorem6} For every  $A\su \N^{<\omega}$, the following are
equivalent:
\begin{itemize}
\item[(i)] $I_{wf}\!\!\upharpoonright \!A$  belongs to
$\mathcal{B}$. \item[(ii)] $I_{wf}\!\!\upharpoonright \!A$ is
Borel.

\item[(iii)] $I_{wf} \not\hookrightarrow I_{wf}\up A$.

\end{itemize}
\end{thr}

Since $I_{wf}$ is a complete co-analytic set and each ideal in
$\base$ is Borel, then it is clear that $(i)\Rightarrow
(ii)\Rightarrow (iii)$.  We will need several auxiliary results
for proving the implication $(iii)\Rightarrow (i)$.

The first step is to show that we can reduce the problem to the
case when $A$ is a tree. Recall that a set $D\su \N^{<\omega}$ is
said to be {\em dense}, if for all $t\in \N^{<\omega}$, there is
$d\in D$ such that $t\preceq d$. The following result is probable
known, we include its proof for the sake of completeness.

\begin{lem}
\label{lemmadensidad} If $D\su \N^{<\omega}$ is dense, then
$I_{wf}\hookrightarrow I_{wf}\up D$.
\end{lem}

\begin{proof}
Fix a bijection $\varphi: \N^{<\omega} \rightarrow \N$ such that
$\varphi(\emptyset)=0$ and $u\preceq t \Rightarrow
\varphi(u)\leq \varphi(t)$. Let $\psi$ be the inverse of
$\varphi$. Inductively, we are going to define a function $h:
\N^{<\omega}\rightarrow D$ such that $u\preceq t \Leftrightarrow
h(u)\preceq h(t)$. If such  function exists, it is easy to see
that $C\notin I_{wf}\Leftrightarrow h[C]\notin I_{wf}$. Therefore,
$h$ is an isomorphism between $I_{wf}$ and $I_{wf}\up D$.

We define $h(\psi(n))$ by induction on $n$. First, fix any
$d_0\in D$ and put $h(\psi(0))=d_0$. Now, suppose $h(\psi(j) )$ has been
defined  for $j\leq k$. Let $u\in \N^{<\omega}$ and
$i\in \N$ such that $\psi(k+1)=u^{\smallfrown} i$. Put
$D_k=\left\langle \{h(\psi(i)) : i\leq k\} \right\rangle \cup
\bigcup\{ \mathcal{N}_{h(t)} : u\prec t  \textrm{ and
}\varphi(t)\leq k \}$. Since $D_k$ is a $\prec$-downward closed
and $D$ is dense, then we can choose $d\in D\cap
\mathcal{N}_{h(u)}$ that is not in $D_k$. Put $h(\psi(k+1))=d$. It
is routine to verify that  $t\prec \psi(k+1) \Leftrightarrow
h(t)\preceq h(\psi(k+1))$.

%If $t\prec \psi(k+1)$ then $t\preceq u$. Using the inductive
%hypothesis we get that $h(t)\preceq h(u)\preceq h(\psi(k+1))$.
%Conversely, if $h(t)\preceq h(\psi(k+1))$ then $h(t)$ and $h(u)$
%are comparable. Notice that it is enough to show that $t\preceq
%u$. Suppose, towards a contradiction that $u\prec t$. From this
%and the fact that $\varphi(t)\leq k$ we get that $h(\psi(k+1))\in
%D_k$ and this is a contradiction.
\fin
\end{proof}

\begin{lem}
\label{lemma6} Let $A\su\N^{<\omega}$ and $T$ be the tree
generated by $A$. If $I_{wf} \hookrightarrow I_{wf}\up T$, then
$I_{wf}\hookrightarrow I_{wf} \up A$.
\end{lem}

\begin{proof} Fix $f:\N^{<\omega} \rightarrow T$
witnessing that  $I_{wf} \hookrightarrow I_{wf}\up T$. We will
define functions $h:\N^{<\omega} \rightarrow \N^{<\omega} $ and
$g:D \rightarrow A$ where $D$ is the range of $h$.
To make the proof easier to read, let us fix a bijection
$\psi: \N \rightarrow \N^{<\omega}$ such
that $\psi(0)=\emptyset$. For $n\in\N$, let $d_n=h(\psi(n))$, so
that we  will have $D=\{d_n: n\in\N\}$. The functions $h$ and $g$
will satisfy the following properties:

\begin{enumerate}
\item[(a)] $\psi(n)\preceq h(\psi(n))$, for all $n\in \N$,

\item[(b)] $f(d_n) \preceq g(d_n)$, for all $n\in \N$,
\item[(c)] $g(d_{n+1}) \npreceq g(d_i)$, for all $i\leq n$.
\end{enumerate}

From  property $(a)$ we get that $D$ is dense in $\N^{<\omega}$
and hence, by Lemma \ref{lemmadensidad}, we get that
$I_{wf}\hookrightarrow I_{wf}\up D$.

We claim that properties $(b)$ and $(c)$ implies that
$I_{wf}\up D\stackrel{g}{\hookrightarrow} I_{wf}\up A$. By $(c)$,
it is clear that $g$ is $1-1$. Now we show that $D \supseteq
C\notin I_{wf}$ iff $g[C]\notin I_{wf}$:

\begin{description}
\item[$(\Rightarrow)$] if $C\su D$ is not in $I_{wf}$, then
$f[C]\notin I_{wf}$ and so there are a sequence $\lambda\in
\N^{\omega}$ and an infinite set $\{c_{l}: l\in \N\}\su C$ such
that for all $l\in\N$, $\lambda\up l \preceq f(c_{l})$. But
$f(c_{l})\preceq g(c_{l})$ for all $l\in\N$. Thus, $g[C]\notin
I_{wf}$.

\item[$(\Leftarrow)$] if $g[C]\notin I_{wf}$ there are a sequence
$\eta\in\N^{\omega}$ and an infinite set $\{c_{l}: l\in \N\}\su C$
such that for all $l\in\N$, $\eta\up l \preceq g(c_{l})$. Notice that
$\{f(c_{l}):\, l\in \N\}$ is infinite (as $f$ is $1-1$). Since $f(c_{l}) \preceq g(c_{l})$ for all
$l\in\N$, then the length of the $f(c_{l})$'s must increase with $l$. Hence, $\eta\in \langle \{ f(c_{l}) : l\in \N \}\rangle$ and so $f[C]\notin I_{wf}$. Therefore $C\notin I_{wf}$, as $f$ is an isomorphism.
\end{description}

In summary, assuming that such functions $h$ and $g$ exist, we have that
$I_{wf}\hookrightarrow I_{wf}\up D$ and $I_{wf}\up
D\hookrightarrow I_{wf}\up A$. Thus, $I_{wf} \hookrightarrow
I_{wf}\up A$.

So it remains to show the construction of  $h$ and
$g$. We will define $h(\psi(n))$ and $g(d_n)$ by induction on $n$.

Pick $a\in A$ such that $f(\emptyset)\preceq a$, and let $h(\psi(0))=d_0= h(\emptyset)=\emptyset$
and $g(d_0)=a$. Note that $\emptyset=\psi(0)\preceq h(\psi(0))$
and $f(d_0)\preceq g(d_0)$. Suppose we have defined
$h(\psi(n))$ and $g(d_n)$ with the desired properties. We claim that
\begin{equation}\label{eqlemma6.2}
(\exists s\in \mathcal{N}_{\psi(n+1)} ) (\exists a\in A) [
f(s)\preceq a  \textrm{ \, and \, } (\forall i\leq n) (a\npreceq
g(d_i))].
\end{equation}
Indeed, since the set $\mathcal{N}_{\psi(n+1)}$ is infinite and
$f$ is $1-1$, we have that $f[\mathcal{N}_{\psi(n+1)}]$ is infinite. In
addition, $f[\mathcal{N}_{\psi(n+1)}] \su \{a\in A : (\exists s
\succeq \psi(n+1)) ( f(s)\preceq a) \} = P$. Hence, $P$ is
infinite. This fact allows us to pick $a\in P$ and $s\in \mathcal{N}_{\psi(n+1)}$
satisfying \eqref{eqlemma6.2}. Finally, we put
$h(\psi(n+1))=d_{n+1}=s$ and $g(d_{n+1})=a$.
\fin
\end{proof}

\bigskip

As mentioned above, when analyzing arbitrary restrictions of $I_{wf}$,
one realizes that the first kind of restriction to understand is the
restriction to trees. To deal with trees we will define a derivative on subsets of
$\N^{<\omega}$. Let $A\su \N^{<\omega}$, we define
\[
A^{\prime}=\{ a\in A :A_{a} \notin I_d \}.
\]
For a successor ordinal, we put $A^{(\beta
+1)}= (A^{(\beta}))^{\prime}$ and for a limit ordinal $\alpha$ we put
$A^{(\alpha)}=\bigcap\limits_{\xi<\alpha} A^{(\xi)}$. The rank of $A$,
denoted  $rk(A)$,  is the first ordinal $\alpha$ such that
$A^{(\alpha)}=A^{(\alpha + 1)}$.

Given a tree $T\su \N^{<\omega}$, we notice that $T^{rk(T)}$ is a tree without
terminal nodes.

\begin{lem}
\label{lemma3}
Suppose $\alpha=\mu +1$ and let $T$ be a tree
with rank $\alpha$ and such that $T^{(\alpha)}=\emptyset$.

\begin{enumerate}
\item[(i)]  If $t\in T$, then $rk(T_{t})=rk(T_{t}\cup \{s : s\preceq t\})$.

\item[(ii)] The tree $H=\{t\in T : rk(T_t)=\alpha\}$ is in $I_d$.
\end{enumerate}
\end{lem}

\begin{proof}
(i) If we let $\gamma=rk(T_t)$, since $T^{(\alpha)}=\emptyset$, then
$T_{t}^{(\gamma)}=\emptyset$. Notice also that $t\in T_{t}^{(\xi)}$ for
all $\xi<\gamma$, whence $\gamma$ is a successor ordinal. Let $\eta$ be so that
$\gamma=\eta + 1$ and put $S=T_{t}\cup \{s : s\preceq t\}$. Then
$S^{(\eta)}=T_{t}^{(\eta)}\cup \{s : s\preceq t\}$ and the result
follows.

(ii) Suppose $H\not\in I_d$. Then there is  $t\in H$ such that
$K=\{n: t^\smallfrown n\in H\}$ is infinite. For every $n\in K$
the set $T_{t^\smallfrown n}$ has rank $\alpha$; so by part $(i)$,
$T^{\ast}_{t^\smallfrown n}=\{s:s\preceq t^\smallfrown n\}\cup
T_{t^\smallfrown n}$ has rank $\alpha$. Consider the tree
$L=\bigcup_{n\in K} T^{\ast}_{t^\smallfrown n}\su H$. We claim
that $t\in L^{(\alpha)}$. In fact, as   $T_{t^\smallfrown
n}^{(\mu)}\su L_t^{(\mu)}$, then ${t^\smallfrown n}\in
L_t^{(\mu)}$ for all $n\in K$ and thus  $L_t^{(\mu)}\not\in I_d$.
Hence $\alpha< rk(L)\leq rk(H)=\alpha$ and this is a
contradiction. \fin
\end{proof}

\begin{lem}
\label{lemma4}
Let $H$ be an infinite tree in $I_d$. For every $s\in H$, let
$P_s\su \N^{<\omega}$ be a set consisting of extensions of $s$ such that $P_s
\cap H=\emptyset$ and $(P_s)_{s\in H}$ is pairwise
disjoint.   Let $P= \bigcup_{s\in H}P_s$ and  $R=H\cup P$.
If $I_{wf}\up P_s \in \mathcal{B}$ for all $s\in H$, then $I_{wf}\up R\in \mathcal{B}$.
\end{lem}

\begin{proof}
We will first proof that $I_{wf}\up P\in
\mathcal{B}$. We claim that
%\[
\begin{equation}
\label{eqlema4}
A\in I_{wf}\up P \Leftrightarrow (\exists
t_0,\dots,t_p\in H)(A\su \bigcup_{i\leq p}P_{t_i}
\textrm{ and } (\forall i\leq p)\; A\cap P_{t_i} \in I_{wf}).
\end{equation}
%\]
In fact, let $A\su \bigcup_{s\in H}P_s$ with $A\in I_{wf}$. By the definition of $I_{wf}$, the tree generated by $A$, denoted $\langle A\rangle $, belongs to $I_{wf}$. Notice also that $H\subseteq \langle \bigcup_{s\in H} P_s \rangle$.  If $A$ meets infinitely many $P_{s}$'s, then $\langle A \rangle $ has infinite
many elements of $H$ (because $P_s\cap P_r=\emptyset$ for $s\neq
r$) and then $\langle A\rangle\notin I_{wf}$ (because $H\in I_d=I_{wf}^\perp$). Thus, $\{s\in
H : A\cap P_s\neq \emptyset\}$ is finite. Put $\{t_0,\dots,
t_p\}=\{s\in H : A\cap P_s\neq \emptyset\}$. Then,
$A\su\bigcup_{i\leq p} P_{t_i}$.  The reverse implication is trivial as $I_{wf}$ is an ideal.

So we have established \eqref{eqlema4}. By Lemma \ref{lemma1} and the fact that $J_s=I_{wf}\up P_{s}\in \mathcal{B}$ for $s\in H$, we have that
\[
I_{wf}\up P \cong \left(\oplus_{s\in H}
J_{s}^{\perp}\right)^{\perp}\in \mathcal{B}.
\]

Finally, since $H\in I_d$ and $H\cap P=\emptyset$ we have that
\[
I_{wf}\!\!\up \! R\; \cong \; I_{wf}\!\!\up\! H\; \oplus \;
I_{wf}\!\!\up \!P\; \cong\; {\FIN} \oplus I_{wf}\!\!\up \! P \in
\mathcal{B}.
\]

\fin
\end{proof}

\begin{lem}
\label{lemma5}
Let $T$ be tree on $\N$ and  $\alpha=rk(T)$.
\begin{enumerate}
\item[(i)] If $T^{\alpha} \neq \emptyset$, then $I_{wf}
\hookrightarrow I_{wf}\up T$.

\item[(ii)] If $T^{\alpha} = \emptyset$, then $I_{wf}\up T \in
\mathcal{B}$.
\end{enumerate}
\end{lem}

\begin{proof}
(i) Suppose $T^{\alpha}\neq \emptyset$. Since
$T^{\alpha}=T^{\alpha + 1}$,  then $T^\alpha$ is a tree without
 terminal nodes. Moreover, given $t\in T^\alpha$, as
$(T^{\alpha})_{t}\notin I_d$, we have that

\begin{equation}\label{eqlema5}
(\forall t\in T^\alpha)(\exists s\succeq t)(A(t,s)=\{ n : s^{\smallfrown}n \in (T^{\alpha})_{t} \} \textrm{ \, is infinite.})
\end{equation}

For each $t\in T^\alpha$, if $s$ is chosen as in \eqref{eqlema5}, we consider the increasing
enumeration of $A(t,s)=\{n_{0}^{(t,s)}<\cdots<n_{k}^{(t,s)}<\cdots\}$.
We can define now an injection  $f:\N^{<\omega} \rightarrow \N^{<\omega}$
by induction on the levels of $\N^{<\omega}$, witnessing $i)$. Lets start defining
$f(\emptyset)=\emptyset$. By \eqref{eqlema5}, there is some $s_{\emptyset}\succeq \emptyset$
so that the set $A(\emptyset,s_{\emptyset})$ is infinite. We define $f$ on the first level
by $f\langle k \rangle=s_{\emptyset}^{\smallfrown} n_{k}^{(\emptyset,s_{\emptyset})}$.
We will define $f$ on one more level to make the construction clear. For every $k\in\N$, choose
$s_{f\langle k \rangle}\succeq f\langle k \rangle$ as in~\eqref{eqlema5}. For each $l\in\N$, we
define $f(\langle k, l\rangle)= s_{f\langle k \rangle} \;  ^{\smallfrown} n_{l}^{(f\langle k \rangle,s_{f\langle k \rangle})}$.

It is easy to check that $f$ is 1-1 and  $A\su \N^{<\omega}$ is
well founded iff $f[A]\in I_{wf}$. Therefore
$I_{wf}\hookrightarrow I_{wf}\up T$.

\medskip

(ii) We will see that $I_{wf}\up T \in \mathcal{B}$, whenever
$T^{\alpha}=\emptyset$, by induction on $\alpha$. If
$T^{\prime}=\emptyset$, then $T\in I_d$ and therefore $I_{wf}\up T
\cong \FIN\in \mathcal{B}$. Now suppose that for every
$\xi<\alpha$ and for every tree $S$ with $rk(S)=\xi$, if
$S^{(\xi)}=\emptyset$ then $I_{wf}\up S \in \mathcal{B}$. Take a
tree $T$ with rank $\alpha$ and such that
$T^{(\alpha)}=\emptyset$. Notice that $\alpha$ cannot be a limit
ordinal, so let $\beta$ be  such that $\alpha=\beta + 1$. Consider
the set
\[
H=\{ t \in T : rk(T_t)=\alpha \}.
\]
By Lemma \ref{lemma3}, $H$ is a tree in $I_d$. For every $s\in H$,
let
\[
M_{s}=\{ n\in \N : rk(T_{s^\smallfrown n})<\alpha \}.
\]
By  Lemma \ref{lemma3}, the tree $T_{s^\smallfrown n}\cup \{u :
u\preceq s^\smallfrown n\}$ has rank smaller than $\alpha$ for
every $s\in H$ and $n\in M_{s}$. Therefore, by the inductive
hypothesis, $I_{wf}\up T_{s^\smallfrown n} \in \mathcal{B}$ for
every $s\in H$ and $n\in M_{s}$. Put $P_{s}=\cup_{n\in M_s}
T_{s^\smallfrown n}$. From Lemma \ref{lemma2}(i) we have that
$I_{wf}\up P_{s}\in \mathcal{B}$. We claim that
\[
T= H \cup \bigcup_{s\in H}P_s.
\]
In fact, only one inclusion needs a proof. Let $t\in T$. If
$rk(T_t)=\alpha$, then $t\in H$. If $rk(T_t)<\alpha$, let $t^{\prime}$ be
the minimal initial segment of $t$ such that
$rk(T_{t^{\prime}})<\alpha$. Notice that $l_0=|t^{\prime}|>0$.
Therefore, the minimality of $t^{\prime}$ implies that $s=t^{\prime}\up(l_0 - 1)\in H$.
Then  $t\in P_{s}$  and $s\in H$.

Thus, $T= H \cup \bigcup_{s\in H} P_{s}$, where $P_s\cap
H=\emptyset$ and, $I_{wf}\up P_{s}\in \mathcal{B}$ for every $s\in
H$. If $H$ is a finite set , then
\[
I_{wf}\up T\cong I_{wf}\up \bigcup_{s\in H} P_s\cong \oplus_{s\in
H} I_{wf}\up P_s\in \mathcal{B}.
\]
If $H$ is infinite,  applying Lemma \ref{lemma4} we get that
$I_{wf}\up T\in \mathcal{B}$. \fin
\end{proof}

\bigskip

\noindent {\bf Proof of Theorem \ref{theorem6}} It remains to show
that (iii) implies (i). Suppose $I_{wf} \not\hookrightarrow
I_{wf}\up A$. Let $T$ be the tree generated by $A$. By Lemma
\ref{lemma6} we have $I_{wf} \not\hookrightarrow I_{wf}\up T$.
Thus, from Lemma \ref{lemma5} we conclude  that $I_{wf}\up T\in
\mathcal{B}$. Hence,  by  Lemma \ref{restricciones}, $I_{wf}\up
A\cong I_{wf}\up (T\cap A)\in \mathcal{B}$. \fin

\bigskip

By taking orthogonal, we get the following immediate consequence
of Theorem \ref{theorem6}.

\begin{coro}\label{coro2} Let $A$ be a subset of $\N^{<\omega}$. Then
$I_{d} \not\hookrightarrow I_{d}\up A$ iff $I_{d}\up A\in
\mathcal{B}$.
\end{coro}

From the previous result and Theorem \ref{theorem4} we get the following.
\begin{coro}
If $I_{wf}\upharpoonright A$ is Borel, then $I_{wf}\upharpoonright
A$ is $F_{\sigma\delta}$.
\end{coro}

\subsection{Borel restrictions of $WO(\Q)$}

In this section we will show a result analogous to Theorem
\ref{theorem6} for the ideal $WO(\Q)$ of the well founded subsets
of $WO(\Q)$. For simplicity, we will write $WO$ instead of
$WO(\Q)$. We first observe that $WO^\perp$ is the ideal of well
founded subsets of $(\Q, <^*)$ where $<^*$ is the reversed order
of $\Q$. In fact, the map $x\mapsto -x$ from $\Q$ onto $\Q$ is an
isomorphism between $WO$ and $WO^\perp$. In particular, $WO$ is a
Fr\'{e}chet ideal.

We recall that  linear order $ (L,<)$ is said to be {\em
scattered}, if it does not contain a order-isomorphic copy of
$\Q$. The main result is the following.

\begin{thr}
\label{restriction of WO} For every $A\su \Q$, the following are
equivalent:
\begin{itemize}
\item[(i)] $A$ is scattered (with the order inherited from $\Q$).
\item[(ii)] $WO\!\!\upharpoonright \!A$ belongs to $\mathcal{B}$.
\item[(iii)]$WO\!\!\upharpoonright \!A$ is Borel.

\item[(iv)] $WO\not\hookrightarrow WO\up A$.

\end{itemize}
\end{thr}

Since $WO$ is a complete co-analytic set (see
\cite[33.2]{Kechris94}) and each ideal in $\base$ is Borel, then
it is clear that $(ii)\Rightarrow (iii)\Rightarrow (iv)$. To see
$(iv)\Rightarrow (i)$, suppose $A\su \Q$ is not scattered. Any
embedding of $(\Q,<)$ inside $(A,<)$ is also an embedding from $WO$
into $WO\up A$. So it only remains to show that $(i)$ implies
$(ii)$. For that end we need to recall a well known result of
Hausdorff about countable scattered orders.

Given a sequence of linear orders $(L_n,<_n)$ over a disjoint
collection of sets $(L_n)_{n\in\N}$, the sum $\sum_{n\in\N} L_n$
is defined as the lexicographical order on $L=\bigcup_n L_n$. That
is to say, for $x,y\in L$, $x <_L y$ iff either $x,y\in L_n$ for
some $n$ and $x<_n y$ or $x\in L_n$ and $y\in L_m$ with $n<m$. The
sum of two (or finitely many) linear orders is defined in a
similar manner. If $L$ is a linear order, then $L^*$ denotes the
reversed order.

We denote by $SC$ the closure of $\{(\N,<)\}$ under the operations
of taking countable or finite sums and reversal of an order. The
collection $SC$ is naturally presented as an increasing union of
of families $SC_\alpha$ with $\alpha<\omega_1$. Where $SC_0$
consists of $\N$,  $\N^*$ and the sums of them $\N+\N^*$ and
$\N^*+\N$. Then $SC_\alpha$ consists of sums of orders of rank
smaller than $\alpha$ and its reversed orders. We say that $L$ has
rank $\alpha$, if $L\in SC_\alpha$ and $L\not\in SC_\beta$ for all
$\beta<\alpha$.

\begin{thr} (Hausdorff \cite{rosenstein82}).
A countable linear order is scattered iff it is isomorphic to an
order in $SC$.
\end{thr}

Notice that if $L\su \Q$, then $(L,<_\Q)^*$ is isomorphic to
$(-L,<_\Q)$ (where $-L=\{-x:\; x\in L\}$). The following simple
observation is the key fact to prove our result.

\begin{lem}
\label{sum=oplus} Let $L\su \Q$.

\begin{itemize}
\item[(i)] If $L$ is order isomorphic to a sum $\sum_{n\in\N}L_n$,
for some disjoint sequence of sets $L_n\su \Q$,  then
$WO\!\!\upharpoonright \!L \cong \bigoplus_n WO\!\!\upharpoonright
\!L_n$.
\item[(ii)] If $L$ is isomorphic to a sum $L_1+L_2$ where $L_1$
and $L_2$ are disjoint subsets of $\Q$, then
$WO\!\!\upharpoonright \!L \cong WO\!\!\upharpoonright \!L_1\oplus
WO\!\!\upharpoonright \!L_2$.
\item[(iii)] $(WO\!\!\upharpoonright \!L)^\perp\cong
WO\!\!\upharpoonright \!{L}^*$.
\end{itemize}
\fin
\end{lem}

\medskip

\noindent {\bf Proof of Theorem \ref{restriction of WO}:} It only
remains to show that $(i)$ implies $(ii)$. This is done by
induction on the scattered order. The base of the induction is
trivial since it is clear that if $L\su\Q$ is order isomorphic to
$\N$, then $WO\!\!\upharpoonright \!L\cong
\mathcal{P}(\mathbb{N})$. The rest follows from Lemma
\ref{sum=oplus}. \fin

\bigskip

\begin{coro}
Neither $I_{wf}$ nor $WO$ is isomorphic to a restriction of the other.
\end{coro}

\proof Recall that $x\mapsto -x$ is an isomorphism between $WO$ and $WO^\perp$. Suppose $WO\cong I_{wf}\up A$. Then $WO\cong WO^\perp \cong I_{wf}^\perp \up A= I_d\up A$, which implies that $WO$ is Borel, as $I_d$ is Borel, and that is a contradiction.

Suppose $ I_{wf}\cong WO\up A$. Then $I_d\cong WO^\perp \up A\cong WO\up (-A)$.  Hence $WO\up (-A)$ is Borel, hence $-A$ is scattered by previous theorem. Thus $A$ is also scattered and thus $WO\up A$ is Borel. Hence $I_{wf}$ is Borel, which is a contradiction.
\fin

\section{Examples of sequential analytic spaces}

As it was explained in the introduction, any ideal can be identified
with a topological space on $X=\N\cup\{\infty\}$ such that the
space is Fr\'{e}chet iff the ideal is Fr\'{e}chet. This idea can be
extended to construct other more complex topological spaces.
Perhaps the most well known  example is Seq (and its variations) which have been studied by several people (see for instance \cite{Franklin68,Louveau72,sirota69,TU}). We will follow the presentation given in
\cite{TU}  where they study a  topology
$\tau_{\mathcal{F}}$ on $\seq$ where ${\mathcal{F}}$ is a filter
over $\N$, such that $(\seq,\tau_{\mathcal{F}})$ is a sequential
space (see the definition below) iff $\mathcal{F}$ is a Fr\'{e}chet
filter (i.e. its dual ideal is Fr\'{e}chet). In fact, they
constructed a family of size bigger than the continuum of
Fr\'{e}chet filters such that the corresponding sequential spaces
$(\seq,\tau_{\mathcal{F}})$ are pairwise non homeomorphic. They
ask if there is an uncountable family of analytic Fr\'{e}chet
filters with the same property. The purpose of this section is to
give a positive answer to that question.

Let us recall that a topological space $X$ is {\em sequential} if
whenever  $A\subseteq X$ is non closed, then there is a sequence
$(x_n)_n$ in $A$ converging to a point not in $A$. Clearly, any
Fr\'{e}chet space is sequential, but the reciprocal is not true.

Let $\mathcal{F}$ be a filter on $\N$ containing the cofinite
sets. Define a topology $\tau_{\mathcal{F}}$ over $\seq$ by
letting a subset $U$ of $\seq$ be open if, and only if, $\{n\in\N
: s^{\smallfrown} n \in U \}\in \mathcal{F}$, for all $s\in U$.
The prototypical sequential space of sequential rank $\omega_1$ is
the well known Arkhangle'ski\v{\i}-Franklin space $Seq$ which
turns out to be homeomorphic to $(\seq,\tau_{\FIN})$. The main
result of this section is that the topological spaces
corresponding to the dual filters of the ideals in $\base$ are
pairwise non-homeomorphic. We need some preliminary results.

\begin{lem} (\cite{TU})\label{lemma8}
Let $\mathcal{F}$ be a filter on $\N$ containing the cofinite
sets. Then
\begin{enumerate}
\item[(i)] $(\seq,\tau_{\mathcal{F}})$ is $T_2$, zero dimensional
and has no isolated points.

\item[(ii)] $(\seq,\tau_{\mathcal{F}})$ is sequential if, and only
if, $\mathcal{F}$ is a Fr\'{e}chet filter.

\item[(iii)] If $(\seq, \tau_{\mathcal{F}})$ is sequential, then
$Seq$ embeds into it as a closed subspace and therefore
$(\seq, \tau_{\mathcal{F}})$ has sequential order $\omega_1$.

\item[(iv)] The space $(\seq, \tau_{\mathcal{F}})$ is homogeneous.

\item[(v)] If $\mathcal{F}$ is Borel, then $\tau_{\mathcal{F}}$ is
Borel (as a subset of $2^{\seq}$).

\end{enumerate}
\end{lem}

We also need the following fact.

\begin{lem}\label{lemma7}
Every ideal $I$ in $\mathcal{B}\setminus \{P_0\}$ is isomorphic to any restriction of itself to a set in its dual filter and $P_0$ is isomorphic to any restriction of itself to an infinite set.
\end{lem}

\begin{proof} The claim about $P_0$ is obvious. Suppose $I$ be an ideal in $\mathcal{B}$ not isomorphic to $P_0$ and $K$  such that $\N\setminus K\in I$. We first treat the case when  $\N\setminus K$ infinite. We have that
\[
I\cong I\up (K \cup  \N\setminus K) \cong I\up K \oplus I \up
(\N\setminus K) \cong I\up K \oplus \, \mathcal{P}(\N){\cong} I\up
K.
\]
The last equivalence follows from Theorem \ref{theorem1}.

Suppose now that $\N\setminus K$ infinite. We will argue by
induction. The basic case is straightforward. Assume that for all $\alpha < \beta$ and for all
cofinite set $L$ with $\N\setminus L \in I$ we have $J\up L \cong J$, whenever
$J\in \{P_{\alpha}, Q_{\alpha}, P_{\alpha} \oplus Q_{\alpha}\}$. Fix a cofinite set $K$ in the dual filter of $I$.

Suppose $I=P_{\beta}$. Consider a sequence of ordinals $(\mu_n)_n$ all less than $\beta$ and consider a partition $(L_n)_n$ of $\N$ such that $I=\bigoplus_{n\in\N} I_n$, where $I\up L_n\cong I_n\cong Q_{\mu_n}$ for every $n\in\N$ (see lemma~\ref{particion-P} and lemma~\ref{sumas-directas}). Since $Q_{\mu_n}\ncong P_0$, we must have that $L_n\notin I_n\cong Q_{\mu_n}$. On the other hand, since $L_n= ((\N\setminus K) \cap L_n)\cup(K\cap L_n)$ and $K$ is cofinite, then it must be the case that $K\cap L_n\notin I_n\cong Q_{\mu_n}$. Then, applying the inductive hypothesis, we get that $I_n\up (K\cap L_n)\cong I_n\cong Q_{\mu_n}$. Therefore,
\[
P_\beta\up K\cong I\up K\cong \bigoplus_{n\in\N} I_n\up(K\cap L_n)\cong \bigoplus_{n\in \N} I_n \cong P_{\beta}.
\]

Suppose now that $I=Q_\beta$. Then $I\up K\cong (I^{\perp}\up K)^{\perp}\cong (P_{\beta}\up K)^{\perp}\cong P_{\beta}^{\perp}\cong Q_{\beta}$.

Finally, suppose $I=I_0\oplus I_2\cong P_\beta \oplus Q_\beta$. Then $I\up K \cong  I_0\up (K\cap A) \oplus I_1 \up (K\cap B)$, where $(A,B)$ is a partition of $\N$ such that  $I\up A\cong I_0\cong P_\beta$ and $I\up B\cong I_1\cong Q_\beta$. Since $\beta\neq 0$ we have that $P_\beta\ncong P_0$; moreover, $Q_\beta\ncong P_0$. Therefore, $A\notin I_0$ and $B\notin I_1$. Since $K$ is cofinite, we have that $K\cap A\notin I_0$ and $K\cap B \notin I_1$. Therefore, $I_0\up (K\cap A)\cong P_{\beta}$ and $I_1 \up (K\cap B)\cong Q_{\beta}$ whence $I\up K\cong I$.
\fin
\end{proof}

\bigskip

We denote by $\mathcal{F}_{\alpha}$ the dual filter of
$P_{\alpha}$, by $\tau_\alpha$ the topology
$\tau_{\mathcal{F}_\alpha}$, and by $\N^{[1]}\su \N^{<\omega}$ the
set of sequences of length $1$.

\begin{pro}
If $\alpha \neq \beta$, then
$(\seq,\tau_\alpha)\ncong(\seq,\tau_\beta)$.
\end{pro}

\begin{proof}
Suppose that $(\N^{<\omega}, \tau_\alpha)\cong (\N^{<\omega},
\tau_\beta)$ and let $h:\N^{<\omega}\rightarrow \N^{<\omega}$ be
an homeomorphism witnessing this fact. By part (iv) of Lemma
\ref{lemma8} we can assume that $h(\emptyset)=\emptyset$. Consider
the sets $A=\{h(\langle n \rangle) : n\in\N \}\cap \N^{[1]}$,
$B=\{s(0):s\in A\}$, and $C=\{h^{-1}(s)(0) : s\in A\}$. Using
Lemma \ref{lemma8}(ii) it is easy to see that $B\in
\mathcal{F}_\beta$, $C\in \mathcal{F}_\alpha$, and
$\mathcal{F}_\alpha \up B \cong \mathcal{F}_\beta \up C$. Thus, by
Lemma \ref{lemma7}, $\mathcal{F}_\alpha\cong\mathcal{F}_\beta$ and
by Lemma \ref{lemma9}, $\alpha=\beta$.

%Lets check the facts stated in the proof: Suppose that $B\notin
%\mathcal{F}$. Then $\N\setminus B \in \mathcal{F}_{\beta}^{+}$.
%Hence, there is an infinite set $D\su \N \setminus B$ in
%$P_{\beta}^{\perp}$. Thus, $(\langle n \rangle)_{n\in D}
%\stackrel{\longrightarrow}{\tau_\beta} \emptyset$. This implies
%that $(h^{-1}\langle n \rangle)_{n\in
%D}\stackrel{\longrightarrow}{\tau_\alpha} \emptyset$. But, $\{
%h^{-1}\langle n \rangle : n\in D \}\cap \N^{[1]}=\emptyset$, a
%contradiction. Thus, $B\in \mathcal{F}_\beta$. In a similar way we
%prove that $C\in\mathcal{F}_\alpha$. It only remains to prove that
%$\mathcal{F}_\alpha\up B \cong \mathcal{F}_\beta\up C$. The
%natural isomorphism is the map $\phi : C \rightarrow B$ given by
%$\phi(n)=h(\langle n \rangle)(0)$. Take a set $K\su C$ such that
%$K\in \mathcal{F}_{\alpha}$ and suppose that $\phi[K]\notin
%\mathcal{F}_\beta$. Then, there is an infinite set $L\su
%B\setminus \phi[K]$ in $P_{\beta}^{\perp}$. Hence, the set
%$M=\{h^{-1}(\langle n \rangle)(0) : n\in L\}\in
%P_{\alpha}^{\perp}$. Therefore, $M\cap (C\setminus K)$ is a finite
%set, which implies that $M$ is almost contain in $K$, a
%contradiction. In a similar way, it can be proved that given a set
%$K\su C$ such that $\phi[K] \in \mathcal{F}_\beta$, then $K\in
%\mathcal{F}_\alpha$.

\fin
\end{proof}

The next result gives a  positive answer to question 6.9 of
\cite{TU}.

\begin{coro}\label{coro4}
There is an uncountable family of pairwise non-homeomorphic
analytic sequential spaces of sequential order $\omega_1$.
\end{coro}

Since the topology of those spaces is analytic (in fact, Borel), then they are homeomorphic to a subspace of $C_p(\N^\N)$ (by Proposition 6.1 of \cite{todoruzca}).

\bigskip

\noindent {\em Acknowledgments:}  We would like to thank the referee for his (her) comments and suggestions which improved the presentation of the results.

\noindent Departamento de Matem\'aticas, Facultad de Ciencias,
Universidad de Los Andes, M\'erida 5101, Venezuela.\\
guevara@ula.ve.\\
\noindent Department of Mathematics, University of Toronto, Toronto, Canada M5S3G3\\
guevara.guevaraparra@mail.utoronto.ca\\

\noindent Escuela de Matem\'aticas, Facultad de Ciencias, Universidad Industrial de
Santander, Ciudad Universitaria, Carrera 27 Calle 9, Bucaramanga,
Santander, A.A. 678, Colombia.\\
cuzcatea@saber.uis.edu.co.


\begin{thebibliography}{99}


\bibitem{AvilesStevo2011}
A.~Avil\'es and S.~Todor\v{c}evi\'{c}.
\newblock Multiple gaps.
\newblock {\em Fund. Math.}, 213:15--42, 2011.

\bibitem{AvilesStevo2012}
{A. Avil{\'e}s}. and {S. Todor{\v{c}}evi{\'c}.}
\newblock Finite basis for analytic strong n-gaps.
\newblock {\em Combinatorica}, 33(4):375--393, 2013.

\bibitem{AvilesStevo2013}
{A. Avil{\'e}s}. and {S. Todor{\v{c}}evi{\'c}.}
\newblock Finite basis for analytic multiple gaps.
\newblock {\em Publications {M}ath{\'e}matiques de l'IH{\'E}S}, 121(1):57--79,
  2014.

\bibitem{AvilesStevo2015}
{A. Avil{\'e}s}. and {S. Todor{\v{c}}evi{\'c}.}
\newblock Compact spaces of the first Baire class that have open finite degree.
\newblock {preprint, arXiv:1512.08363v1 [math.GN] 28 Dec 2015},
  2015.

\bibitem{Franklin68}
A.~V. Arkhanglel'ski{\~{\i}} and S.~P. Franklin,
\newblock Ordinal invariants for topological spaces,
\newblock {\em Mich. Math. J.}, 15:313--320, 1968.

\bibitem{Debs87} G. Debs. Effective properties in compact sets
of Borel functions. {\em Mathematika}, 34 $(1987)$ $64-68$.

\bibitem{Debs2009} G. Debs. Borel extractions of converging sequences
in compact sets of Borel functions. {\em Journal of Mathematical
Analysis and its Aplications}. 350 (2009) 731-744.

\bibitem{DK} P. Dodos and V. Kanellopoulos, On pairs of
definable orthogonal families. {\em Illinois J. Math}. Volume 52,
Number 1 (2008), 181-201.

\bibitem{Farah2000}
I.~Farah.
\newblock {\em Analytic Quotiens}, volume 148 of {\em Memoirs of the AMS}.
\newblock Providence, Rhode Island, 2000.


\bibitem{GU2009}
S.~Garc\'ia-Ferreira and C.~Uzc\'ategui.
\newblock Subsequential filters.
\newblock {\em Topology and its Applications}, 156(18):2949 -- 2959, 2009.

\bibitem{Garcia-Rivera2013}
S.~Garc\'ia-Ferreira and J.E. Rivera-G\'omez.
\newblock Ordering {F}r\'echet-{U}rysohn filters.
\newblock {\em Topology and its Applications}, 163: 128--141, 2014.

\bibitem{Kechris94}
A.~S. Kechris.
\newblock {\em Classical Descriptive Set Theory}.
\newblock Springer-Verlag, 1994.

\bibitem{Krawczyk92}
A.~Krawczyk.
\newblock On the {R}osenthal compacta and analytic sets.
\newblock {\em Proc. Amer. Math. Soc.}, 115:1095--1100, 1992.


\bibitem{Louveau72}
A. Louveau.
\newblock Sur un article de {S}. {S}irota.
\newblock {\em Bull. Sci. Math. (2)}, 96:3--7, 1972.

\bibitem{Mathias77} A.R.D. Mathias, Happy families, {\em Annals Math.
Logic}, 12 (1977), 59-111.

\bibitem{rosenstein82}
J.~Rosenstein.
\newblock {\em Linear orderings}, volume~98.
\newblock Academic Press, 1982.

\bibitem{si98} P. Simon, \textit{A hedghog in the product},
Acta. Univ. Carolin. Math. Phys. \textbf{39} (1998), 147-153.

\bibitem{Simon2008}
P.~Simon.
\newblock A countable {F}r\'echet-{U}rysohn space of uncountable character.
\newblock {\em Topology and its Applications}, 155(10):1129 -- 1139, 2008.


\bibitem{sirota69}
S.~M. Sirota.
\newblock A product of topological groups, and extremal disconnectedness.
\newblock {\em Mat. Sb. (N.S.)}, 79 (121):179--192, 1969.

\bibitem{Solecki1999}
S.~Solecki.
\newblock Analytic ideals and their applications.
\newblock {\em Annals of Pure and Applied Logic}, 99(1-3):51--72, 1999.

\bibitem{To} S. Todor\v{c}evi\'{c}. Analytic gaps. {\em Fundamenta
Mathematicae}. 150 (1996), pp. 55-66.

\bibitem{Todor99}
S.~Todor\v{c}evi\'c.
\newblock Compacts sets of the first baire class.
\newblock {\em Journal of the AMS}, 12(4):1179--1212, 1999.

\bibitem{Todor2010}
S.~Todor\v{c}evi\'c.
\newblock {\em Introduction to Ramsey spaces}.
\newblock Annals of Mathematical Studies 174. Princeton University Press, 2010.

\bibitem{todoruzca}
S.~Todor\v{c}evi\'c and C.~Uzc\'ategui,
\newblock Analytic topologies over countable sets,
\newblock {\em Top. and its Appl.}, 111(3):299--326, 2001.

\bibitem{TU} S. Todor\v{c}evi\'{c} and C. Uzc\'{a}tegui. Analytic $k-$spaces.
{\em Topology and its Applications}. 146-147 (2005) 511-526.


\end{thebibliography}
\end{document}